\documentclass[preprint]{imsart}
\usepackage{bbm}
\usepackage{amsmath,amssymb,graphicx,amsthm}
\usepackage{epsfig,color, enumitem}

\startlocaldefs

\def\PP{{\mathbb P}}
\def\EE{{\mathbb E}}

\newcommand{\nbR}{\mathbbm{R}}
\newcommand{\nbN}{\mathbbm{N}}

\newcommand{\nbC}{\mathbbm{C}}
\newcommand{\nbu}{\mathbbm{1}}

\newcommand{\nbP}{\mathbbm{P}}
\newcommand{\nbZ}{\mathbbm{Z}}

\newcommand{\nbE}{\mathbbm{E}}

\newtheorem{theorem}{{\bf Theorem}}
\newtheorem{lemm}[theorem]{{\bf Lemma}}
\newtheorem{cor}[theorem]{{\bf Corollary}}
\newtheorem{remark}[theorem]{\it Remark}
\newtheorem{prop}[theorem]{\bf Proposition}
\newtheorem{example}[theorem]{\it Example}

\endlocaldefs

\begin{document}

\begin{frontmatter}

\title{High order expansions for renewal functions and applications to ruin theory}
\runtitle{Expansions for renewal functions}


\author{\fnms{Dombry} \snm{Cl\'ement}\ead[label=e1]{Clement.Dombry@univ-fcomte.fr}}
\and
\author{\fnms{Rabehasaina} \snm{Landy}\ead[label=e2]{Landy.Rabehasaina@univ-fcomte.fr}}
\address{Universit\'e de Bourgogne Franche-Comt\'e, Laboratoire de Math\'ematiques, \\ CNRS UMR 6623,
UFR Sciences et Techniques,\\
16 route de Gray,
25030 Besan\c{c}on cedex, France.\\
\printead{e1}\\
\printead{e2}}

\runauthor{Dombry, Rabehasaina}

\begin{abstract}
A high order expansion of the renewal function is provided under the assumption that the inter-renewal time distribution is light tailed with finite moment generating function $g$  on a neighborhood of $0$. This expansion relies on complex analysis and is expressed in terms of the residues of the function $1/(1-g)$. Under the assumption that $g$ can be extended into a meromorphic function on the complex plane and some technical conditions, we obtain even an exact expansion of the renewal function. An application to risk theory is given where we consider high order expansion of the ruin probability for the standard compound Poisson risk model. This precises the well known Cr\'amer-Lundberg approximation of the ruin probability when the initial reserve is large.
\end{abstract}

\begin{keyword}[class=MSC]
\kwd[Primary ]{60K05}
\kwd{}
\kwd[; secondary ]{60K10}
\end{keyword}

\begin{keyword}
\kwd{renewal process}
\kwd{renewal function}
\kwd{ ruin probability}
\kwd{ Cram\'er-Lundberg approximation}
\kwd{compound Poisson model}
\end{keyword}

\end{frontmatter}


\section{Introduction}
Let $(X_k)_{k\in \nbN}$ be an i.i.d. sequence of non negative random variables with common cumulative  distribution $F$.
The arrival times $(S_n)_{n\in\nbN}$ are defined by  $S_0=0$ and, for $n\geq 1$, $S_n=\sum_{k=1}^n X_k$. 
We consider the counting process $N$ defined by
\[
N(x):= \sum_{n\geq 0} 1_{\{S_n\leq x\}},\quad x\ge 0,\\
\]
and the associated renewal function
\[
U(x):= \EE[N(x)]=\sum_{n=0}^\infty F^{\ast n}(x),\quad x\ge 0.
\]
The renewal Theorem states that if the inter-arrival distribution $F$ has a finite first moment $\mu:=\EE(X_1)$, then 
$$U(x)\sim \frac{x}{\mu}\quad \mbox{as }x\to+\infty.$$
Recall that the distribution $F$ is called lattice if $F$ is supported by $h\mathbb{N}=\{0,h,2h,\ldots\}$ for some mesh $h>0$. It is well known (see for example Asmussen \cite[Proposition 6.1]{A03}) that if $F$ has a finite second moment $\mu_2:=\EE(X_1^2)$, then
\begin{equation}\label{eq:firstorder}
 U(x)=\left\{\begin{array}{cc}\frac{x}{\mu}+\frac{\mu_2+\mu}{2\mu^2}+o(1)& \mbox{if } F\mbox{ is lattice} \\
              \frac{x}{\mu}+\frac{\mu_2}{2\mu^2}+o(1)& \mbox{otherwise} 
             \end{array}\right.
 \quad \mbox{as }x\to+\infty.
\end{equation}
When $F$ has infinite first or second moment, Teugels \cite{Teugels} provides asymptotics for the renewal function $U$ under some regular variation conditions. In this paper, we focus on the case when $F$ is light-tailed and we assume that $X_1$ has some finite exponential moment so that 
\begin{equation}\label{def_R}
 R:=\sup\left\{r\geq 0;\ \int_0^\infty e^{rx}dF(x)<\infty\right\} >0.
\end{equation}
Then the moment generating function
\[
g(z):=\int_0^\infty e^{zx}dF(x)
\]
is well defined and holomorphic on 
\[
S_R=\{z\in\mathbb{C}, \Re(z)<R\}.
\]
Using complex analysis, Stone \cite{Stone} proved that under the strong non lattice condition
\begin{equation}\label{eq:stone_strong_lattice}
\limsup_{\theta \to\pm \infty} \left|\frac{1}{1-g(i\theta)}\right|<\infty,
\end{equation}
there exists some $r>0$ such that 
\begin{equation}\label{expansion_Stone}
U(x)=\frac{x}{\mu}+\frac{\mu_2}{2\mu^2}+o(e^{-rx}) \quad  \mbox{as }x\to\infty.
\end{equation}
Up to now, relatively few results concern either expansions or closed form expressions for $U(x)$. It appears that the only case where closed form expansions are available are when $F$ is Matrix exponential distributed, in which case an expression of $U(x)$ is given in Asmussen and Bladt \cite{AB}. Mitov and Omey \cite{Mitov_Omey} provide heuristics on many terms asymptotics of $U(x)$ that are verified on the already known cases. However, as the authors point out, those interesting expansions are only given formally and are not proved. Other expansions are available in \cite[Theorem 4]{DS} in the context of potential densities of L\'evy processes in terms of the associated L\'evy jump distribution. The approach by Stone \cite{Stone} for obtaining Expansion (\ref{expansion_Stone}) is mainly based on complex analysis and proved fruitful for obtaining expansions involving survival functions of random sums, see Blanchet and Glynn \cite{BG}. The approach in \cite{Stone} was later generalized for spread out distributions in \cite{Stone3}; note that the generalization of our results to spread out distribution is not available in the present paper because the main technical assumption that enable us to obtain higher expansions (namely, Assumption (\ref{condA}) thereafter), which is stronger than the non lattice condition (\ref{eq:stone_strong_lattice}), is not verified for spread out distributions. This is not really surprising, as it is on the other hand known that a spreadout distribution is strongly non lattice (see e.g. Proposition 1.6 p.189 of \cite{A03}), which is why many results in renewal theory that hold for strongly non lattice distributions also hold for spreadout distributions. 

We consider in this paper both the lattice and the non lattice cases and, in view of Equation \eqref{eq:firstorder}, we introduce the function 
\[
 v(x):= \left\{\begin{array}{cc}U(x)-\frac{x}{\mu}-\frac{\mu_2+\mu}{2\mu^2}& \mbox{if } F\mbox{ is lattice} \\
              U(x)-\frac{x}{\mu}-\frac{\mu_2}{2\mu^2}& \mbox{if } F\mbox{ is non-lattice}
             \end{array}\right..
\]
Following Stone's approach but with more detailed computations, we obtain higher order expansions for the function $v$  both in the lattice and non-lattice cases (Theorems \ref{thm_expansion_v_lattice} and \ref{thm_expansion_v} respectively). When $g$ has a meromorphic extension to the whole complex plane and under some technical conditions, we are even able to provide exact expansions for $v$ (Corollary \ref{cor_expansion_exact}). 

The paper is structured as follows. Section 2 presents our results on high order expansions for the renewal function as well as some examples. Section 3 is devoted to applications to ruin theory: we provide asymptotics of the ruin probability in the setting of continuous or discrete time risk processes and also consider a two dimensional risk process. Proofs are gathered in Sections 4 and 5.

\section{High order expansions for the renewal function}\label{sec:renewal}
\subsection{Main results}\label{sec:main_results}
In the sequel, the solutions of the equation $g(z)=1$ play a major role. Note that $0$ is the unique real solution in $S_R$ and that other solutions satisfy $\Re(z)\geq 0$ and come in pairs, i.e. if $z$ is a solution then so is $\bar z$.

We consider first the lattice case and we assume without loss of generality that the mesh of the distribution is equal to $h=1$, i.e.  $F$ is supported  by $\mathbb N$. In this case, the moment generating function $z\mapsto g(z)$ is $2i\pi$-periodic on $S_R=\{z\in\mathbb{C}, \Re(z)<R\}$ and we introduce the fundamental domain 
$S_{R}^f=\{ z\in \nbC;\ \Re (z)<R, -\pi\leq \Im(z)\leq \pi\}$.  
\begin{theorem}\label{thm_expansion_v_lattice}
Let $R_0\in (0,R)$ be such that Equation $g(z)=1$ has no root verifying $\Re(z)=R_0$. Let us denote by $z_0=0, z_1,\ldots, z_N$ the solutions of this equation  in $S_{R_0}^f$. Then, $v(k)$ has the asymptotic expansion 
\begin{equation}\label{expansion_v_lattice}
v(k)=\sum_{j=1}^N \mathrm{Res}\left(\frac{e^{-kz}}{(e^z-1)(1-g(z))};z_j\right) + o(e^{-R_0 k}),\quad k\to+\infty,\ k\in\nbN,
\end{equation}
where the notation $\mathrm{Res}(f(z);z_j)$ denotes the residue of the meromorphic function $f$ at pole $z_j$.\\ 
If $z_j$ is a simple zero of $g-1$, i.e.  $g'(z_j)\neq 0$, then the $j$-th term in \eqref{expansion_v_lattice} has the simple form
\[
-\frac{e^{-kz_j}}{(e^{z_j}-1)g'(z_j)}.
\]
\end{theorem}

Next we consider the case when $F$ is non-lattice. It is well known that a distribution is non-lattice if and only if $\theta=0$ is the unique real solution of the equation $g(i\theta)=1$. 
We will need here the following stronger  technical assumption: for all $R_0< R$,
\begin{equation}\label{condA}
\limsup_{\theta\to\pm\infty}\sup_{0\leq r\leq R_0} \left|\frac{1}{1-g(r+i\theta)}\right| <\infty.
\end{equation}
Assumption \eqref{condA} is stronger than the strong lattice condition (\ref{eq:stone_strong_lattice}) from Stone \cite{Stone}. It is however not too restrictive and satified by a large class of distributions as shown by the following propositon.
\begin{prop}\label{prop_A}
Suppose that the distribution $F$ is absolutely continuous with respect to the Lebesgue measure. Then Assumption \eqref{condA} is satisfied for all $R_0<R$.
\end{prop}

Our main result in the non-lattice case is the following theorem.
\begin{theorem}\label{thm_expansion_v}
Consider a strictly non-lattice distribution $F$ satisfying  assumption \eqref{condA}. Then for all $R_0\in(0,R)$, the equation $g(z)=1$ 
has a finite number of solutions in $S_{R_0}=\{z\in\mathbb{C};\ \Re(z)<R_0\}$ that we denote by $z_0=0, z_1,\ldots,z_N$. Then, supposing that $R_0$ is such that there is no solution to that equation verifying $\Re(z)=R_0$, $v(x)$ has the asymptotic expansion 
\begin{equation}\label{expansion_v}
v(x)=\sum_{j=1}^N \mathrm{Res}\left(\frac{e^{-xz}}{z(1-g(z))};z_j\right) + o(e^{-R_0x})\quad \mbox{as } x\to+\infty. 
\end{equation}
If $z_j$ is simple zero of $g-1$, i.e.  $g'(z_j)\neq 0$, then the $j$-th term in \eqref{expansion_v} has the simple form
\[
-\frac{e^{-xz_j}}{z_jg'(z_j)}.
\]
\end{theorem}

It is worth noting that Theorems \ref{thm_expansion_v_lattice} and \ref{thm_expansion_v} can be extended to obtain asymptotics of higher order, i.e. of order $e^{-rx}$ with $r>R$,  if we assume that the moment generating function $g$ has a meromorphic extension to $S_{\bar R}$ for some $\bar R>R$.  
Theorems \ref{thm_expansion_v_lattice} and \ref{thm_expansion_v} and their proofs extend in a straightforward way, but not Proposition \ref{prop_A}. 
In the case when the moment generating function $g$ has a meromorphic extension to the whole complex plane, i.e. $\bar R=+\infty$,
it is even possible under some technical assumption to get an exact expansion for the renewal function.

\begin{cor}\label{cor_expansion_exact}
Assume that $g(z)$ has a meromorphic extension to the whole complex plane. 
\begin{itemize}
 \item In the lattice case, we assume furthermore that
 \begin{equation}\label{eq:cond_lattice}
  \liminf_{r\to +\infty} \sup_{-\pi\leq\theta\leq \pi} \left|\frac{1}{e^{r}(1-g(r+i\theta))}\right|=0.
 \end{equation}
We denote by $z_0=0,z_1,\ldots,z_N$ (possibly $N=+\infty$) the solution of $g(z)=1$ in the fundamental domain  
 $S^f=\{ z\in \nbC;  -\pi\leq \Im(z)\leq \pi\}$ with $\Re(z_0)\leq \Re(z_1)\leq\cdots$. Then, we have the exact expansion
 \begin{equation}\label{exact_expansion_v_lattice}
 v(k)=\sum_{j=1}^N \mathrm{Res}\left(\frac{e^{-kz}}{(e^z-1)(1-g(z))};z_j\right),\quad k\geq 0.
 \end{equation}
\item In the non-lattice case, we suppose that assumption \eqref{condA} holds for all $R_0>0$. Let $x>0$. Let us furthermore suppose that
\begin{itemize}
\item One has an infinite number of roots $(z_n)_{n\in\nbN}$ of Equation $g(z)=1$, all of which are simple, real, and such that the following series converges:
\begin{equation}\label{exact_expansion_abs_convergence}
\sum_{j=1}^\infty \left|\mathrm{Res}\left(\frac{e^{-xz}}{z(1-g(z))};z_j\right)\right|= \sum_{j=1}^\infty \left|\frac{e^{-xz_j}}{z_j g'(z_j)} \right|<+\infty .
\end{equation}
\item The following holds:
\begin{equation}\label{eq:cond_nonlattice}
  \liminf_{r\to +\infty} \sup_{\theta\in\mathbb{R}}\left|\frac{1}{ r(1-g(r+i\theta))}\right|=0.
 \end{equation}
\end{itemize}
Then, we have the exact expansion
\begin{equation}\label{exact_expansion_v_non_lattice}
 v(x)=\sum_{j=1}^\infty \mathrm{Res}\left(\frac{e^{-xz}}{z(1-g(z))};z_j\right).
 \end{equation}
\end{itemize} 
\end{cor}
\begin{remark}\label{comment_referee}
{\rm One of the assumptions of Corollary \ref{cor_expansion_exact} in the non-lattice case is that there is an {\it infinite} number of roots of Equation $g(z)=1$. One may wonder what happens in the case when those roots are in finite number $N$. In fact, one shows in this case, and thanks to Condition (\ref{eq:cond_nonlattice}), that the moment generating function $g(z)$ is rational, in which case a finite expansion for $v(x)$ can be deduced almost straightforwardly.}
\end{remark}
In the previous results, a kind of dichotomy arises between the lattice and non-lattice cases. 
Interestingly, a unified statement can be deduced for the renewal mass function or the renewal density function in the non-lattice and lattice case respectively.
\begin{cor}\label{cor_expansion_exact_2}
Under the same assumptions as Corollary \ref{cor_expansion_exact}:
\begin{itemize}
 \item In the lattice case, the renewal measure has mass function
 \[
 u(k)=-\sum_{j=0}^N \mathrm{Res}\left(\frac{e^{-kz}}{1-g(z)};z_j\right),\quad k\geq 0. 
 \]
 As a particular case, if all the poles are simple, 
 \[
 u(k)=\sum_{j=0}^N \frac{e^{-kz_j}}{g'(z_j)},\quad k\geq 0.
 \]
 \item In the non-lattice case, the renewal measure has density function
 \begin{equation}
 u(x)=- \sum_{j=0}^N\mathrm{Res}\left(\frac{e^{-xz}}{1-g(z)};z_j\right),\quad x>0. \label{renewal_density_polynome}
 \end{equation}
 As a particular case, if all the poles are simple, 
 \begin{equation}
 u(x)=\sum_{j=0}^N \frac{e^{-xz_j}}{g'(z_j)},\quad x> 0.\label{renewal_density_polynome_simple_pole}
 \end{equation}
\end{itemize}
\end{cor}
To conclude this subsection, we present an informal argument leading to (and motivating) Expansion (\ref{expansion_v}) and that may lead to some better comprehension of proof of Theorem \ref{thm_expansion_v} given in Section \ref{sub_section_proof_theo_expansion}. One verifies, using Fubini, that the Laplace transform of $U(.)$ is
\begin{multline*}
\int_0^\infty e^{-xz} U(x) dx = \int_0^\infty e^{-xz} \left[ \sum_{k=0}^\infty \nbP(S_k\le x)\right] dx=\sum_{k=0}^\infty \nbE\left[ \int_{S_k}^\infty e^{-xz} dx\right]\\
= \sum_{k=0}^\infty \frac{g(-z)^k}{z}= \frac{1}{z(1-g(-z))},
\end{multline*}
so that, the inversion formula for the Laplace transform leads formally to
\begin{equation}\label{Mellin_transf}
U(x)=\frac{1}{2i\pi}\int_{c-i\infty}^{c+i\infty}\frac{e^{xz}}{z(1-g(-z))}dz
\end{equation}
where $c>0$ is such that all singularities of $\frac{e^{xz}}{z(1-g(-z))}$ are on the left of vertical line $c+i\nbR$. The right hand side integral of (\ref{Mellin_transf}) does not necessarily converge, however let us suppose that this is the case. The poles of $\frac{e^{xz}}{z(1-g(-z))}$ satisfying $\Re(z)>-R_0$ are $-z_0=0$,..., $-z_N$, and one can verify that the residue of $\dfrac{e^{xz}}{z(1-g(-z))}$ at $-z_0=0$ is $ \dfrac{x}{\mu}+ \dfrac{\mu_2}{2\mu^2}$. Thus, using a rectangular contour and the theorem of residues leads to the following
\begin{eqnarray}
U(x)&=& \sum_{j=0}^N \mathrm{Res}\left( \frac{e^{xz}}{z(1-g(-z))};-z_j\right) + \frac{1}{2i\pi}\int_{-R_0+i\infty}^{-R_0-i\infty}\frac{e^{xz}}{z(1-g(-z))}dz \nonumber\\
&=& \dfrac{x}{\mu}+ \dfrac{\mu_2}{2\mu^2} + \sum_{j=1}^N \mathrm{Res}\left( \frac{e^{-xz}}{z(1-g(z))};z_j\right)\nonumber\\
&& + \frac{1}{2i\pi}\int_{-R_0+i\infty}^{-R_0-i\infty}\frac{e^{xz}}{z(1-g(-z))}dz. \label{Mellin_transf2}
\end{eqnarray}
The last integral is an $o(e^{-R_0x})$, so that one would obtain Expansion (\ref{expansion_v}). However, the main failing points in this reasoning are first that the integral in (\ref{Mellin_transf}) is not convergent, and that the contour argument leading to (\ref{Mellin_transf2}) is more delicate than it seems. The convergence issue will be solved by introducing a gaussian kernel (an idea already introduced by Stone \cite{Stone, Stone2}) which, by inversion, will make the corresponding integral converge, see Step 2 in the proof in forthcoming Section \ref{sub_section_proof_theo_expansion}. The contour argument will involve Assumption (\ref{condA}), which will enable two of the pieces of the contour to vanish in the proof, see again Step 2 in Section \ref{sub_section_proof_theo_expansion} as well as the corresponding Figure \ref{figure_contour}.

\subsection{Examples}
We provide some examples that illustrate the results above.
\begin{example}{\rm
In the lattice case, we consider the negative binomial distribution with parameters $p\in (0,1)$ and $n\geq 1$ defined by
\[
F(dx)=\sum_{k\geq 0}\binom{k+n-1}{k} p^k(1-p)^{n}\delta_k(dx),\quad p\in (0,1), n\geq 1.
\]
Its moment generating function is given by
\[
g(z)=\left(\frac{1-p}{1-pe^z}\right)^n,\quad \Re(z)<R=-\log p,
\]
so that
\[
\frac{1}{1-g(z)}=\frac{(1-pe^z)^n}{(1-pe^z)^n-(1-p)^n}
\]
defines a meromorphic function on $\mathbb{C}$. The poles are the solutions of
\[
e^{z}=\frac{1-(1-p)e^{2i\pi\frac{j}{n}}}{p},\quad 0\leq j\leq n-1.
\]
For $n=1$ the only pole in the fundamental domain $S_{\bar R}^f$ is $z_0=0$ so that Theorem \ref{thm_expansion_v_lattice} implies  $v(x)=o(e^{-rx})$ for all $r>0$. 
In the general case $n\geq 1$, there are exactly $n$ poles in the fundamental domain $S^f$ given by $z_j=r_ke^{i\alpha_j}$ with
\[
r_j=\log\left|\frac{1-(1-p)e^{2i\pi\frac{j}{n}}}{p}\right| \quad \mbox{and}\quad \alpha_j=\mathrm{arg}\left(1-(1-p)e^{2i\pi\frac{j}{n}}\right).
\]
Furthermore, Assumption \eqref{eq:cond_lattice} is easily satisfied (comparison between an exponential and a power growth) so that Corollaries  
\ref{cor_expansion_exact} and \ref{cor_expansion_exact_2} apply. We obtain that the renewal measure has mass function
\[
 u(k)=\sum_{j=0}^{n-1} \frac{1-pe^{z_j}}{np(1-p)e^{z_j}}e^{-kz_j} \quad \mbox{with}\quad  e^{z_j}=\frac{1-(1-p)e^{2i\pi\frac{j}{n}}}{p}. 
\]
}\end{example}

\begin{example} {\rm In the non-lattice case let us consider {\it Matrix exponential distributions} with parameters $(\alpha,T)$, where $\alpha$ is an $\nbR^{1\times (N+1)}$ probability row vector and $T$ is an $\nbR^{(N+1)\times (N+1)}$ subintensity matrix. The definition and principal properties of Phase type distributions can be found in Asmussen and Albrecher \cite[Chapter IX]{AA}. The moment generating function $g(z)$ is a rational function given by
\begin{equation}
g(z)= \alpha (-zI-T)^{-1}s
\label{mgf_PH}
\end{equation}
where $s:=-T e$ and $e=(1,\ldots,1)^T$ (see \cite[Theorem 1.5]{AA}). The equation $g(z)=1$ is a polynomial equation with $N+1$ solutions $z_0=0,\ldots,z_N$ and Condition \eqref{condA} is satisfied with $R_0=+\infty$ and Assumption \eqref{eq:cond_nonlattice} is easily satisfied (comparison between an exponential and a power growth). From Corollaries \ref{cor_expansion_exact} and \ref{cor_expansion_exact_2}, we obtain a closed formula for the  renewal distribution $U$ and the renewal density $u$. On the other hand, Asmussen and Bladt \cite[Theorem 3.1]{AB} provide the simple expression
\begin{equation}
u(x)=\alpha e^{(s\alpha+T)x} s,\quad x>0 .
\label{renewal_density_PH}
\end{equation}
Let us check that this formula agrees with Corollary \ref{cor_expansion_exact_2} in the case when the roots $z_0,\ldots,z_N$ are simple. It is easy to check that the $-z_j$'s are exactly the eigenvalues of the matrix $s\alpha+T$. Denoting by $P$ the $(N+1)\times (N+1)$ matrix made up with the eigenvectors $v_1,\ldots,v_j\in \nbR^{N+1}$ corresponding to eigenvalues $z_0,\ldots,z_N$, we have
$$
s\alpha+T= P \Delta P^{-1}\quad\mbox{with}\quad \Delta:= \mbox{Diag}(-z_j,\ j=0,...,N).
$$
Denoting by $J_k$ the $\nbR^{(N+1)\times (N+1)}$ matrix with $1$ at the $(k,k)$th position and $0$ elsewhere, we compute
\begin{equation}\label{PH_1}
u(x)=\alpha e^{(s\alpha+T)x} s= \alpha P e^{\Delta x} P^{-1}s= \sum_{j=1}^{N+1} e^{-z_{j-1} x} \alpha P J_j P^{-1}s .
\end{equation}
On the other hand, we deduce from (\ref{mgf_PH}) 
\begin{multline*}
g'(z)=\alpha (-zI-T)^{-2}s= \alpha (-zI+s\alpha - P \Delta P^{-1})^{-2}s\\
= \alpha P (-zI+P^{-1} s\alpha P-  \Delta )^{-2} P^{-1}s.
\end{multline*}
Hence, in order to prove that  (\ref{PH_1}) and (\ref{renewal_density_polynome}) agree, we need to prove
\[
\alpha P J_j P^{-1}s=\frac{1}{g'(z_{j-1})}\quad \mbox{for all}\ j=1,\ldots,N+1
\]
or equivalently
\[
\alpha P J_j P^{-1}s \times \alpha P (-z_{j-1}I+P^{-1} s\alpha P-  \Delta )^{-2} P^{-1}s =1.
\]
This can be easily verified with elementary algebra (the relation $J_j(\Delta+z_{j-1}I)=0$ is useful).
}\end{example}

\begin{example}\label{example_uniform}
 {\rm  Let us consider the simple case of the {\it uniform distribution} on $[0,1]$, i.e. $X\sim {\cal U}[0,1]$. In that case $g(z)=\frac{e^z-1}{z}$ and
   the equation $g(z)=1$ is equivalent to $e^z=z+1$.  The solutions are $z_j=-W_j(-e^{-1})-1$, $j\in\nbZ$, where $W_j(.)$ is the $j$th generalized Lambert function. 
   Proposition \ref{prop_A} together with Theorem \ref{thm_expansion_v} provide an asymptotic expansion of $v$. Using the relations 
   $z_{-j}=\bar z_j$ and $g'(z_j)=1$, we obtain, for all $N\geq 1$, 
  \[
  v(x)=-2\sum_{j=1}^N \Re\left(\frac{1}{z_j}e^{-xz_j}\right)+o(e^{-r_Nx})\quad \mbox{as}\ x\to+\infty,  
  \] 
with $z_j=-W_j(-e^{-1})-1$ and $r_N=\Re(z_N)$. As $N\to+\infty$, $r_N\to+\infty$ so that the expansion has arbitrary high order. 
} \end{example}

At this point, there still lacks an example of distribution $X$ such that a meromorphic extension $g(z)$ exists, Equation $g(z)=1$ admits an infinite number of solution and infinite expansion (\ref{exact_expansion_v_non_lattice}) holds. An example of such an infinite expansion of $v(x)$ will be given in upcoming Example \ref{example_meromorphic}, in the context of meromorphic L\'evy processes.

\section{Application to ruin theory}\label{sec:ruin}
As an application of Theorem~\ref{thm_expansion_v}, we provide asymptotic expansions for the ruin probability in risk theory. We consider both a continuous setting (compound Poisson risk process) and a discrete setting (binomial risk process). Estimation of the ruin probability in a two-dimensional model is also investigated.

\subsection{Ruin theory in continuous time}\label{subsec_risk_continuous}
We consider the  following classical continuous time risk process
$$
R_t^x=x + ct - \sum_{k=1}^{N_t} Z_k=x+Y_t,\quad t\ge 0,
$$
with $\{N_t,\ t\ge 0 \}$ a Poisson process with intensity $\alpha>0$ and $(Z_k)_{k\in \nbN}$ an  i.i.d. sequence of non-negative random variables with common  distribution $G$ and finite expectation $m$, independent from $\{N_t,\ t\ge 0 \}$. Such a process models the capital of an insurance company with premium rate $c>0$, initial reserve $x$, and incoming claims $(Z_k)_{k\in \nbN}$, see e.g. Asmussen and Albrecher \cite{AA}. We define $\bar G=1-G$ the tail function and assume the following on the moment generating function
\begin{equation}\label{condB}
s\mapsto \int_0^\infty e^{sx}dG(x)\mbox{ is finite for all } s>0.
\end{equation}
We are interested in the ruin probability
\begin{equation}\label{def_ruin_proba}
\psi(x):=\PP\left( \inf_{t\ge 0}R_t^x <0\right),\quad x\ge 0.
\end{equation}
and its asymptotic when the initial reserve $x$ is large. It is well known that $\psi(x)<1$ if and only if the safety loading is positive, i.e. 
\begin{equation}\label{safety_loading}
\EE(Y_1)=c-\alpha m >0.
\end{equation}
In the asymptotic analysis, a key role is played by the Lundberg equation
\begin{equation}\label{Lundberg}
\int_0^\infty e^{zy}\frac{\alpha}{c} \bar G(y)\mathrm{d}y=1,\quad z\in\mathbb{C}.
\end{equation}
Under Assumption (\ref{safety_loading}), this equation restricted to real numbers admits a unique real solution denoted by $\kappa>0$. The Lundberg inequality states that 
\[
\psi(x)\leq e^{-\kappa x} \quad \mbox{for all}\ x>0,
\]
while the Cram\'er-Lundberg approximation provides the asymptotic behavior as $x\to+\infty$
\begin{equation}\label{eq:CL}
\psi(x)\sim C e^{-\kappa x}\quad \mbox{with}\quad C=\frac{c-\alpha m}{\EE\left(Ze^{\kappa Z} \right)-c}.
\end{equation}
We provide high order asymptotic expansions for the ruin probability $\psi(x)$. Similar considerations as well as exact expansions have been proved with different methods by Kuznetsov and Morales \cite{KM} for a so called meromorphic risk process  and by Roynette et al. \cite{RVV}. 

Using the fact that $\kappa>0$ solves the Lundberg Equation \eqref{Lundberg}, one can define the probability measure $F$ on $[0,+\infty)$ by
\[
F(dx):= e^{\kappa x}\frac{\alpha}{c} \bar{G}(x)dx.
\]
The moment generating function
\[
 g(z)=\int_0^\infty e^{zx}F(dx),\quad z\in\mathbb{C}
\]
is well defined and holomorphic on the complex plane. For future reference, note that the Lundberg equation \eqref{Lundberg} is equivalent to
\begin{equation}\label{Lundberg2}
 g(z-\kappa)=1,\quad z\in\mathbb{C}.
\end{equation}

Using a renewal equation solved by the ruin probability function $\psi$ and the asymptotic behavior of the renewal function provided by Theorem \ref{thm_expansion_v}, 
we can deduce an asymptotic expansion for $\psi(x)$ as $x\to+\infty$.
\begin{theorem}\label{prop_exp_ruin}
Assume conditions \eqref{condB} and \eqref{safety_loading} are satisfied.  Let $r>0$ be fixed and $z_0=0, z_1,\ldots,z_N$ the solutions of $g(z)=1$ in $S_{r}=\{z\in\mathbb{C};\ \Re(z)\leq r\}$. 
Then, the ruin probability $\psi(x)$ has the asymptotic expansion
\begin{equation}
\psi(x)=  \sum_{j=0}^N \mathrm{Res} \left(\frac{(\alpha m-c)e^{-x(z+\kappa)}}{c(1-g(z))(z+\kappa)};z_j\right) + o(e^{-(r+\kappa)x})\quad \mbox{as } x\to +\infty.\label{exp_ruin_new_continuous}
\end{equation}
If $z_j$ is a simple zero of $g-1$, i.e. $g'(z_j)\neq 0$, then the $j$-th term in \eqref{exp_ruin_new_continuous} has the simple form
\begin{equation}
-\frac{(\alpha m-c)}{cg'(z_j)(z_j+\kappa)}e^{-x(z_j+\kappa)}=\frac{c-\alpha m}{\alpha \EE\left( Z e^{(z_j+\kappa)Z}\right)-c}\, e^{-(z_j+\kappa) x}.\label{g_prime_not_zero_continuous_risk}
\end{equation}
\end{theorem}
The term $j=0$ of the asymptotic expansion \eqref{exp_ruin_new} is exactly the Cramer-Lundberg approximation \eqref{eq:CL}. 
\begin{example}\label{ex_stop_loss}
 {\rm Let us consider the case where the claims are of the form  $X=\min(V,d)$ where $d>0$ and $V$ has an exponential distribution with parameter $\lambda>0$. This models a reinsurance scenario where a reinsurance company covers  the excess of claim above $d$ only, i.e. according to a stop loss contract with priority $d>0$. In that case, Lundberg Equation (\ref{Lundberg}) reads
\begin{equation}
e^{(z-\lambda)d}=1+ \frac{c}{\alpha}(z-\lambda),\quad z\in\nbC, \label{eq_stop_loss}
\end{equation}
and the Lundberg exponent is $\kappa=\lambda$. 
Solutions $z_j$ satisfy  
$$
z_j=-\frac{\alpha}{c}-\frac{1}{d}W_j\left(-\frac{\alpha d}{c}e^{-\frac{\alpha d}{c}}\right),\quad j\in \nbZ,
$$
where $W_j(.)$ is the $j$th generalized Lambert function. It is easy to check that $$g'(z_j)=d+\frac{d\alpha/c-1}{z_j+\kappa-\lambda}\neq 0,$$ so that Theorem \ref{prop_exp_ruin} entails the asymptotic expansion 
(\ref{exp_ruin_new_continuous}) with $j$th term given by (\ref{g_prime_not_zero_continuous_risk}). Since claims here are bounded, the expression of the ruin probability can in fact be made more precise. Indeed $\psi(x)=1-E[Y_1]W(x)$ where $W(x)$ is the so-called scale function associated to L\'evy process $\{Y_t,\ t\ge 0 \}$ (see Expression (8.7) p.215 of \cite{Kyprianou}). The expression of $W(x)$ is available in Theorem 3 of \cite{KP} as an infinite series, it yields that (\ref{exp_ruin_new_continuous}) can in fact be written as an infinite series, i.e.
$$
\psi(x)=  -\sum_{j=0}^\infty \frac{(\alpha m-c)}{cg'(z_j)(z_j+\kappa)}e^{-x(z_j+\kappa)}.
$$
}\end{example}
\begin{remark}\label{rem_risk_meromorph}{\rm It is worth comparing Theorem~\ref{prop_exp_ruin} with the results of Kuznetsov and Morales  \cite{KM}.
They consider a so called meromorphic risk process  $\{R_t,\ t\ge 0\}$, which amounts to assume that the claims $Z_k$'s have density
$$
\frac{\PP[Z_k \in dx]}{dx}=\sum_{m=1}^\infty b_m e^{-\rho_m x},\quad x\ge 0,
$$
for some positive coefficients $(b_m)_{m\ge 1}$ and increasing sequence $(\rho_m)_{m\ge 1}$ satisfying  $\rho_m\to +\infty$. Corollary 1  of \cite{KM} states that the Laplace exponent $\Lambda(z):=\log \EE\left( e^{z Y_1} \right)$ of the L\'evy process $\{R_t,\ t\ge 0\}$ admits a meromorphic extension on $z\in\nbC$ and that all the solutions of the (extended) Lundberg equation \eqref{Lundberg2} are real, negative and simple.
Furthermore, denoting these solutions by $(-\zeta_n)_{n\ge 1}$, the ruin probability \eqref{def_ruin_proba} has  expansion
\begin{equation}
\psi(x)= \sum_{n=1}^\infty \frac{-\EE[Y_1]}{ \Lambda'(-\zeta_n)} e^{-\zeta_n x}
\label{expansion_KM}
\end{equation}
In this framework, condition \eqref{condB} is not satisfied but we check below that these results are still consistent with Theorem~\ref{prop_exp_ruin}. The Laplace exponent satisfies 
\begin{equation*}
\Lambda(z)=cz+\alpha\left[ \EE \left( e^{-z Z_1}\right)-1\right]
\quad \mbox{and}\quad \Lambda'(z)=c-\alpha\EE \left( Z_1e^{-z Z_1}\right).
\end{equation*}
Elementary computations reveal that
\begin{multline}\label{compute_rem_KM}
g(z)= \frac{\alpha}{c} \frac{1}{z+\kappa}\EE \left( e^{(z+\kappa) Z}-1\right)=\dfrac{\Lambda(-z-\kappa)}{c(z+\kappa)}+1\\
\mbox{and}\quad g'(z)=\frac{\alpha}{c}\frac{\EE \left( Ze^{(z+\kappa)Z}\right)}{z+\kappa}-\frac{\alpha}{c} \frac{\EE \left( e^{(z+\kappa) Z}-1\right)}{(z+\kappa)^2}.
\end{multline}
One can check that the solutions of $g(z)=1$ satisfy $z_n+\kappa=-\zeta_n$ and that 
\[
 g'(z_j)=\frac{1}{z_j+\kappa}\left[\frac{\alpha}{c} \EE \left( Ze^{(z_j+\kappa)Z}\right)-1\right],
\]
so that Expansion  (\ref{exp_ruin_new_continuous}), with corresponding terms given by (\ref{g_prime_not_zero_continuous_risk}), corresponds to the $N+1$ first terms of (\ref{expansion_KM}).}
\end{remark}

\begin{remark}{\rm
With some more effort, Theorem~\ref{prop_exp_ruin} can be extended to more general Gerber-Shiu functions, e.g. of the form
$$
\psi(x,\theta,b,a):=\EE_x\left( e^{-\theta \tau}\nbu_{\{R_{\tau^-}\ge b;\, \underline{R}_{\tau}\ge a; \,\tau<+\infty \} }\right)
$$
where $\theta$, $b$, $a$ are non negative, $\tau:=\inf\{t\ge 0|\ R_t<0 \}$ is the ruin time of the risk process   and $\underline{R}_{t} =\inf_{0\le s\le t} R_s$ is the running minimum at time $t$, 
see Theorem 2.8 of \cite{RVV} as well as Theorem 1 of \cite{KM} for example of such expansions. For ease of presentation, we  stick in this paper to $\psi(x)$ as defined by (\ref{def_ruin_proba}).} 
\end{remark}

\begin{example}\label{example_meromorphic}
 {\rm We give an example of an infinite expansion in the non lattice case of $v(x)$ as in Corollary \ref{cor_expansion_exact}. Conditions (\ref{exact_expansion_abs_convergence}) and (\ref {eq:cond_nonlattice}) may look hard to verify in practice.
 To exhibit such an $X$, we again use the theory of meromorphic L\'evy processes. As in Remark \ref{rem_risk_meromorph}, we pick spectrally negative process $\{Y_t,\ t\ge\}$, $Y_t=ct-\sum_{k=1}^{N_t} Z_k$ where $\{N_t,\ t\ge\}$ is a Poisson process with intensity $\alpha>0$, such that Laplace exponent is of the form
$$
\Lambda(z)=\tilde{\mu}z+ z^2\sum_{m=1}^\infty \frac{b_m}{\rho_m^2(\rho_m+z)},
$$
for some $\tilde{\mu}>0$, where sequences of positive real numbers $ (b_m)_{m\in\nbN^*}$ and (strictly) increasing $(\rho_m)_{m\in\nbN^*}$ are such that series $\sum_{m=1}^\infty \frac{b_m}{\rho_m}$ converges so that L\'evy process $\{Y_t,\ t\ge\}$ is indeed a compound Poisson process, see (3.2) in \cite{KM}. We will additionally suppose that sequence $(\rho_m)_{m\in\nbN^*}$ grows like a polynomial (in addition to being increasing), i.e. there exists $a\ge 1$ such that
\begin{equation}
\rho_m\sim C m^a,\quad m\to\infty 
\label{cond_convergence_rho_m}
\end{equation}
for some $C>0$. Remembering that $G$ is the cdf of the $Z_k$'s, we then consider r.v. $X$ with descending ladder height distribution  of L\'evy process $\{Y_t,\ t\ge 0\}$, with corresponding moment generating function
$$g(z):=\int_0^\infty e^{(\kappa+z)y}\frac{\alpha}{c}\bar G(y)\mathrm{d}y,$$
where $\kappa>0$ is solution to Lundberg equation (\ref{Lundberg}), see Relation (5.7) p. 87 of \cite{AA}. We proceed to show that an infinite expansion for the corresponding function $v(x)$ is available. The sole condition for this expansion is (\ref{cond_convergence_rho_m}), which is not too stringent and covers a wide range of processes. The relation between $g(z)$ and $\Lambda(z)$ is given by (\ref{compute_rem_KM}). This has two important consequences. The first one is that $z\mapsto g(z)$ is meromorphic, as $\Lambda(z)$ is. The second one is that $z$ is a solution  to $g(z)=1$ iff $\Lambda(-z-\kappa)=0$. By Properties (v) and (vi) of \cite{KM} (see also Theorem 1 (7) of \cite{KKP}), one deduces that roots $(z_n)_{n\in\nbN}$ are {\it real non negative} and verify
\begin{equation}\label{roots_ordered}
z_0=0<\rho_1-\kappa<z_1<\rho_2-\kappa< z_2<...
\end{equation}
We now turn back to Conditions (\ref{exact_expansion_abs_convergence}) and (\ref {eq:cond_nonlattice}). We start by (\ref{exact_expansion_abs_convergence}). We compute from (\ref{compute_rem_KM})
\begin{equation}
g(z)=-\frac{\tilde{\mu}}{c} + 1+ (z+\kappa)\sum_{m=1}^\infty \frac{b_m}{c\rho_m^2 (\rho_m-z-\kappa)},
\label{expression_g_infinite_expansion}
\end{equation}
hence, for all $j\in\nbN$, $g'(z_j)=\sum_{m=1}^\infty \frac{b_m}{c\rho_m (\rho_m-z_j-\kappa)^2}$, which happens to be positive. We now write
$$
z_j^2g'(z_j)\ge   \sum_{m=1}^j \frac{b_m}{c\rho_m \left(1+\frac{\rho_m-\kappa}{z_j}\right)^2}= \sum_{m=1}^\infty \frac{b_m}{c\rho_m \left(1+\frac{\rho_m-\kappa}{z_j}\right)^2}\nbu_{[m\le j]}.
$$
Using (\ref{roots_ordered}) and the dominated convergence theorem, one easily shows that $\sum_{m=1}^\infty \frac{b_m}{c\rho_m \left(1+\frac{\rho_m-\kappa}{z_j}\right)^2}\nbu_{[m\le j]}\longrightarrow \sum_{m=1}^\infty \frac{b_m}{c\rho_m }$ as $j\to\infty$. One then deduces from the above inequality that
\begin{equation}\label{lim_inf_example_infinite_expansion}
  \liminf_{j\to \infty}z_j^2g'(z_j)\ge  \sum_{m=1}^\infty \frac{b_m}{c\rho_m }>0.
\end{equation}  
Now, (\ref{cond_convergence_rho_m}) and (\ref{roots_ordered}) implies that $\sum_{m=1}^\infty e^{-xz_m}$ is a convergent series for all $x>0$ which, combined with (\ref{lim_inf_example_infinite_expansion}), implies the convergence (\ref{exact_expansion_abs_convergence}).\\
 We now prove (\ref {eq:cond_nonlattice}), by establishing that $\lim_{n\to\infty} \frac{1}{ r_n(1-g(r_n+i\theta))}=0$ with $r_n:=\rho_n-\kappa$. Using (\ref{expression_g_infinite_expansion}) and $g(0)=1$ implies for all $\theta\in\nbR$
\begin{eqnarray*}
1-g(r_n+i\theta)&=& g(0)-g(r_n+i\theta)\nonumber\\
&=& \sum_{m=1}^\infty \frac{b_m}{c\rho_m^2} \left[\frac{\kappa}{\rho_m-\kappa} -\frac{\rho_n+i\theta}{\rho_m-\rho_n-i\theta}\right]  \\
&=& \sum_{m=1}^\infty \frac{b_m}{c\rho_m^2} \left[\frac{\kappa}{\rho_m-\kappa} +1 -\rho_m \frac{\rho_m-\rho_n+i\theta}{(\rho_m-\rho_n)^2+\theta^2}\right]\\
&:=&R_n(\theta)+iI_n(\theta).
\end{eqnarray*}
Let us set $R_{n,1}(\theta):= \sum_{m=1}^\infty \frac{b_m}{c\rho_m^2} \left[\frac{\kappa}{\rho_m-\kappa} +1\right]+ \sum_{m=1}^n \frac{b_m}{c\rho_m^2}\rho_m \frac{\rho_n-\rho_m}{(\rho_m-\rho_n)^2+\theta^2}$ and $ R_{n,2}(\theta):=\sum_{m=n+1}^\infty \frac{b_m}{c\rho_m^2}\rho_m \frac{\rho_n-\rho_m}{(\rho_m-\rho_n)^2+\theta^2}$. As $(\rho_n)_{n\in\nbN}$ is increasing one gets the following inequalities
\begin{eqnarray*}
R_{n,1}(\theta)&\ge & \sum_{m=1}^\infty \frac{b_m}{c\rho_m^2} \left[\frac{\kappa}{\rho_m-\kappa} +1\right]:=\xi>0,\\
0\ge R_{n,2}(\theta) &\ge & \sum_{m=n+1}^\infty \frac{b_m}{c\rho_m^2} \rho_m \frac{\rho_n-\rho_m}{(\rho_m-\rho_n)^2}=\sum_{m=n+1}^\infty \frac{b_m}{c\rho_m}\frac{1}{\rho_n-\rho_m}:=\chi_n,
\end{eqnarray*}
so that the real part of $1-g(r_n+i\theta)$ verifies the inequality
\begin{equation}\label{expression_g_infinite_expansion2}
|R_n(\theta)|\ge  |R_{n,1}(\theta)| - |R_{n,2}(\theta)|= R_{n,1}(\theta) + R_{n,2}(\theta)\ge  \xi + \chi_n.
\end{equation}
(\ref{cond_convergence_rho_m}) entails that $|\chi_n|\le \frac{1}{\rho_{n+1}-\rho_n}\sum_{m=n+1}^\infty \frac{b_m}{c\rho_m}\sim \frac{1}{C an^{a-1}}\sum_{m=n+1}^\infty \frac{b_m}{c\rho_m}\longrightarrow 0$ as $n\to\infty$.
One then deduces that $\xi + \chi_n>0 $ for $n$ large enough, and
$$
\sup_{\theta\in\nbR} \frac{1}{r_n |1-g(r_n+i\theta)|}\le \sup_{\theta\in\nbR} \frac{1}{r_n |R_n(\theta)|}\le \frac{1}{r_n(\xi + \chi_n)}\longrightarrow 0,\quad n\to \infty,
$$
proving (\ref {eq:cond_nonlattice}). Hence infinite expansion (\ref{exact_expansion_v_non_lattice}) holds.

} \end{example}

\subsection{Skip free random walks on $\nbZ$}\label{subsec_risk_discrete}
Quite unlike its continuous time counterpart, risk theory in discrete time seems to have been less studied. 
We refer to \cite{SLG} for an overview of such processes, as well as  \cite[Chapter XVI]{AA}. This type of process is but a skip free random walk, i.e. a random walk with at most unit upward movement, and is in fact studied in many fields of applied probability. We consider here the so-called binomial discrete time risk model defined by
\[
R_n=x+n-\sum_{j=1}^n Z_j=x+Y_n,\quad n\in \nbN
\]
where $x\in\nbN$ is the initial reserve, the premium rate is assumed w.l.o.g. to be equal to $1$, the claims 
$(Z_j)_{j\in \nbN}$ form an i.i.d. sequence taking values in $\nbN$. We let $m_1=\mathbb{E}[Z_1]$ and assume that $m_1\in (0,1)$.

The discrete ruin probability is defined by
\begin{equation}
\psi(x):=\PP\left( \inf_{n\in\nbN\setminus \{0 \}}R_n \le 0\right),\quad x\in\nbN .
\label{def_ruin_proba_discrete}
\end{equation}
This corresponds to the probability that a $\nbZ$-valued random walk starting from $x\in \nbN$ eventually becomes nonpositive.
The condition $m_1\in(0,1)$ ensures that the random walk  has a positive drift so that $\psi(x)<1$.

A closed form expression for $\psi(x)$ may be found in \cite{Gerber}, however this expression requires computing an infinite number of convolutions of distribution of $Z_1$.
We are here interested in finding a simple expansion of $\psi(x)$ as $x\to \infty$. Similarly to Condition \eqref{condB} in the continuous case, we assume here that
\begin{equation}\label{condB'}
\mbox{the moment generating function }\EE\left[e^{sZ}\right]\mbox{ is finite for all } s>0. 
\end{equation}
In this discrete setting, the Lundberg equation writes 
\begin{equation}
\sum_{k=0}^\infty e^{zk}\mathbb{P}[Z>k]=1,\quad z\in\nbC. \label{eq_Lundberg_discrete}
\end{equation}
Restricted to the real numbers, this equation has, thanks to convexity of mean generating function of $Z$, a unique solution $\kappa>0$. We define the probability mass function $f$ defined by
$
 f(k)=e^{\kappa k}\mathbb{P}[Z>k],\quad k\in\mathbb{N},$
with moment generating function
\[
 g(z)=\sum_{k=0}^\infty e^{z k}f(k)=\frac{1-\EE(e^{(z+\kappa)Z})}{1-e^{z+\kappa}},\quad z\in\nbC.
\]
The following theorem provides an  asymptotic expansion of $\psi(x)$ as $x\to \infty$, $x\in\nbN$ and is the discrete analog of Theorem \ref{prop_exp_ruin}.
\begin{theorem}\label{thm:ruin_discret}
Assume conditions \eqref{condB'} holds. Let $r>0$ be fixed and $z_0=0, z_1,\ldots,z_N$ the solutions to Equation $g(z)=1$ in $S_{r}^f=\{z\in\mathbb{C};\  \Re(z)<r, \ -\pi\leq \Im(z)\leq \pi\}$. 
Then, the ruin probability $\psi(x)$ has the asymptotic expansion
\begin{equation}
\psi(x)=  -\sum_{j=0}^N \mathrm{Res} \left[\frac{1}{1-g(z)} \frac{m-g(z)e^{\kappa+z}}{1-e^{\kappa+z}} e^{-(z+\kappa)x };z_j\right] + o(e^{-(r+\kappa)x}) \quad \mbox{as } x\to \infty.\label{exp_ruin_new}
\end{equation}
If $z_j$ is a simple zero of $g-1$, i.e. $g'(z_j)\neq 0$, then the $j$-th term in \eqref{exp_ruin_new} has the simple form
\begin{equation}
\frac{m-e^{\kappa+z_j}}{e^{\kappa+z_j}-\EE\left(Z e^{(\kappa+z_j)Z}\right)}\, e^{-(z_j+\kappa) x}.\label{expression_term_ruin_discrete}
\end{equation}
\end{theorem}

\subsection{A two dimensional ruin problem}
We consider a two dimensional ruin problem motivated by applications in reinsurance. The capitals of two insurance companies are modeled by the risk processes 
\begin{equation}
R_t^j=x_j + c_j t - \sum_{k=1}^{N^j_t} Z^j_k=x_j+Y^j_t,\quad t\ge 0,\quad j=1,2,
\label{model_2dim}
\end{equation}
where, $x_j\ge 0$, $c_j>0$ are the respective initial reserves and premium rates, $\{N^j_t,\ t\ge 0 \}$ are Poisson processes with intensities $\alpha_j>0$, and $(Z^j_k)_{k\in \nbN}$ are the corresponding claims with mean $m^j$. For each $j=1,2$, independence between $\{N^j_t,\ t\ge 0 \}$ and  $(Z^j_k)_{k\in \nbN}$ is assumed. However no  independence is required between processes $\{R^1_t,\ t\ge 0 \}$ and $\{R^2_t,\ t\ge 0 \}$. We suppose that the mean drifts $\EE(Y_j)=c_j-\alpha_j m^j$, $j=1,2$, are positive, and then define the eventual ruin probabilities for each company
$$
\psi_j(x_j):=\PP\left( \inf_{t\ge 0}R^j_t <0\right),\quad j=1,2.
$$
We also consider the probability that (at least) one of the companies is eventually ruined
$$
\psi_{\small \mbox{or}}(x_1,x_2):= \PP\left( \inf_{t\ge 0}R^1_t <0\mbox{ or } \inf_{t\ge 0}R^2_t <0 \right),
$$
i.e. the probability that the two dimensional process $ \{(R^1_t,R^2_t),\ t\ge 0 \}$ exits the first quadrants $[0,+\infty)^2$. 

We are interested here in the asymptotics of $\psi_{\small \mbox{or}}(x_1,x_2)$ as $(x_1,x_2)$ tend to infinity along a fixed direction $x_2/x_1=q\in (0,+\infty)$. 
We refer to \cite{APP, Rabe} for related results that concern light tailed claims, or \cite{Biard, HJ} for models featuring heavy tailed claims.

For $j=1,2$, letting $g_j(z):=\EE(e^{z Z^j_1})$, we suppose that Equation $g_j(z)=1$ has solutions $z_0^j=0$, $z_1^j$, $\overline{z_1^j}$ in $S_r=\{z\in\mathbb{C},0\le \Re(z)<r\}$ for some $r>0$, and that those $z_1^j$, $\overline{z_1^j}$ are simple zeros of $g-1$. Thus $\psi_j(x_j)$  has  the following $2$ terms expansion from Theorem \ref{prop_exp_ruin} 
\begin{equation}
\psi_j(x_j)=C_0^j e^{-\kappa_j x_j} + \Re\left[C_1^j e^{-(\kappa_j+z^1_j) x_j}\right] + \varepsilon_j(x_j)e^{-(r+\kappa_j)x_j},\quad x_j\to +\infty,\quad j=1,2,
\label{two_dim_expansions}
\end{equation}
where $\varepsilon_j(x_j)\longrightarrow 0$ as $x_j\to+\infty$, and
$$
C_k^j:=\frac{c_j-\alpha_j m^j}{\alpha_j \EE\left( Z^j e^{(z_j+\kappa)Z^j}\right)-c_j},\quad k=0,1,\quad j=1,2.
$$
The main result of this subsection is the following theorem.
\begin{prop}\label{Prop_2dim}
A two term asymptotic for $\psi_{\small \mbox{or}}(x_1,x_2)$ as $(x_1,x_2)$ tend to infinity along the fixed direction $x_2/x_1=q\in (0,+\infty)$, is given by
\begin{equation}
\psi_{\small \mbox{or}}(x,qx)=\Re\left[D_0 e^{-d_0(q)x}\right]+\Re\left[D_1 e^{-d_1(q)x}\right] + \eta_q(x)e^{-\Re(d_1(q))x},\quad x\to \infty,
\label{two_terms_two dim}
\end{equation}
where $d_0(q)>0$, $\Re(d_1(q))>d_0(q)$, $0\le \limsup_{x\to \infty}|\eta_q(x)|\le 1$ and where   $x\mapsto \Re\left[D_0 e^{-d_0(q)x}\right]$ and $x\mapsto \Re\left[D_1 e^{-d_1(q)x}\right]$ are the two first dominant functions among  $x\mapsto C_0^1 e^{-\kappa_1 x}$, $x\mapsto C_0^2 e^{-q\kappa_2 x}$, $x\mapsto \Re\left[C_1^1 e^{-(\kappa_1+z^1_1) x}\right]$ and $x\mapsto \Re\left[C_1^2 e^{-q(\kappa_2+z^2_1) x}\right]$.
\end{prop}
Four different cases occur in the  asymptotic  described in Proposition \ref{Prop_2dim}, depending on the asymptotic direction $\vec{u}=(1,q)$ : 
\begin{multline*}
\Re(\kappa_1+z_1^1)>q\kappa_2,\ \kappa_1>q\kappa_2>\Re(\kappa_1+z_1^1),\\
q\kappa_2>\kappa_1 > q\Re(\kappa_2+z_1^2),
\mbox{ or}\ q\Re(\kappa_2+z_1^2)>\kappa_1 .
\end{multline*}
To each case corresponds a different two terms expansion for $\psi_{\small \mbox{or}}(x,qx)$ as summarized in Figure \ref{fig_2dim1}.
Proposition \ref{Prop_2dim} generalizes the one term expansion given in Theorem 3 of \cite{APP}. The last term in (\ref{two_terms_two dim}) is only  $O(e^{-\Re(d_1(q))x})$ but the condition $\limsup_{x\to \infty}|\eta_q(x)|\leq 1$ provides information on how fast this term tends to $0$.

\begin{figure}[!h]
\begin{center}
   \includegraphics[width=0.65\textwidth]{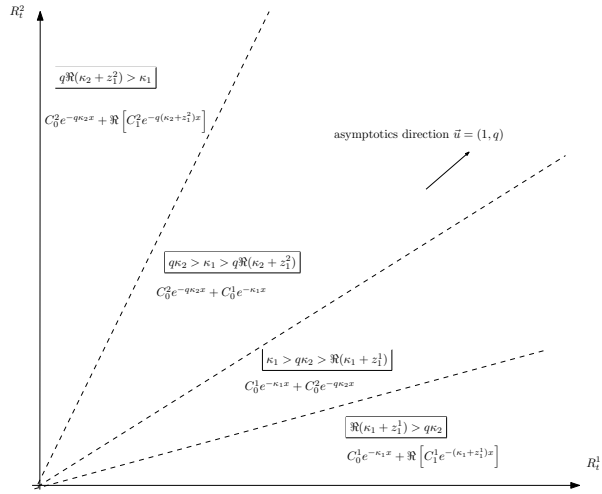} 
\end{center}
\caption{Two term asymptotic expansions on the four different regions.}
\label{fig_2dim1}
\end{figure}

\begin{example}{\rm Let us consider the Stop Loss contract scenario with priority $d>0$. We assume that $\{ R_t^1,\ t\ge 0\}$ is the capital of an insurance company with claims $(Z_n^1)_{n\in \nbN}$ distributed as $\min(V,d)$ where  $V$ is exponentially distributed with parameter $\lambda$. The second risk process $\{ R_t^2,\ t\ge 0\}$ corresponds to the capital of a reinsurance company which covers the excess of claims with priority $d$, i.e. claims  $(Z_n^1)_{n\in \nbN}$ are distributed as $(V-d)^+$, as described in Example \ref{ex_stop_loss}. 
In that case, the two risk processes  $R^1_t$ and $R^2_t$ are dependent. Because of the memoryless property of the exponential distribution, the $Z_n^2$'s are exponentially distributed with parameter $\lambda$ and the intensity of the Poisson process $\{ N_t^2,\ t\ge 0\}$ is given by
$$
\alpha_2=\alpha_1 \PP(V>d)=\alpha_1 e^{-\lambda d}.
$$
As the $Z_n^2$'s are exponentially distributed, 
\begin{equation}
\psi_2(x_2)= C_0^2 e^{-\kappa_2 x_2},\quad \kappa_2=\lambda- \alpha_2/c_2,\quad C_0^2=\frac{\alpha_2}{c_2\lambda},
\label{expansion_expo_claims}
\end{equation}
see Corollary 3.2 p.78 of \cite{AA}. The two terms expansion for $\psi_1(x_1)$ is given by
$$
\psi_1(x_1)= C_0^1 e^{-\kappa_1 x_1} +  \Re\left[C_1^1 e^{-(\kappa_1+z^1_1) x}\right] + o(e^{-(r+\kappa_1)x_1})
$$
where $z_1^1$ is the solution to Equation (\ref{eq_stop_loss}) with smallest real part $\Re(z_1^1)>\kappa_1$ and $r>\Re(z_1^1)$.
Note that even though $\{ R_t^1,\ t\ge 0\}$ and $\{ R_t^2,\ t\ge 0\}$ are correlated, Proposition \ref{Prop_2dim} may be applied and, since $C_1^2=0$, only three cases occur :
\[
\Re(\kappa_1+z_1^1)>q\kappa_2,\quad \kappa_1>q\kappa_2>\Re(\kappa_1+z_1^1)\quad \mbox{and}\quad q\kappa_2>\kappa_1.
\]
This is summarized in Figure \ref{fig_2dim2}.
\begin{figure}[!h]
\begin{center}
   \includegraphics[width=0.65\textwidth]{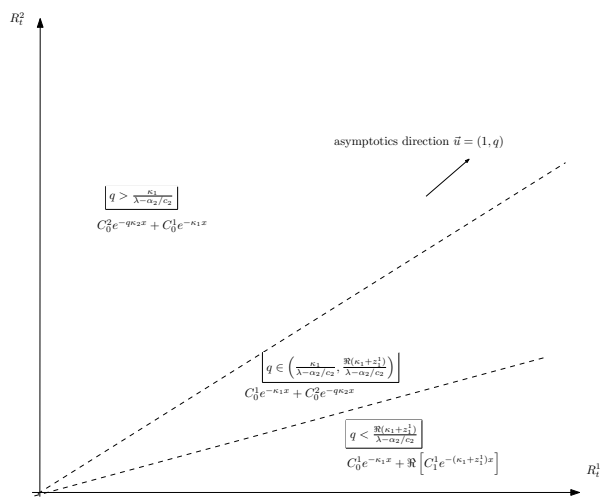}
\end{center}
\caption{Two term asymptotic expansions for the stop loss model. }
\label{fig_2dim2}
\end{figure}
}
\end{example}

\section{Proofs for section \ref{sec:renewal}}
\subsection{Proof of Theorem \ref{thm_expansion_v_lattice}}
For the proof of  Theorem \ref{thm_expansion_v_lattice}, we need the following lemma.
\begin{lemm}\label{lemma_weak_conv}
Let $f:[-\pi,\pi]\longrightarrow \nbC$ be a continuous function satisfying $f(0)\in\nbR$ and $\Im(f(\theta))=O(\theta)$ as $\theta\to 0$. Then we have the following convergence as $ r\to 1-$:
\begin{equation}
\int_{-\pi}^\pi   \Re\left(f(\theta)\frac{1}{1-r g(i\theta)}\right)d\theta\longrightarrow \int_{-\pi}^\pi   \Re\left(f(\theta)\frac{1}{1- g(i\theta)}\right)d\theta + \pi\frac{f(0)}{\mu}.
\label{weak_conv}
\end{equation}
\end{lemm}
The above lemma is akin to the preliminary result  of Stone \cite[p.330]{Stone2}, see also Breiman \cite{Breiman} and Feller and Orey \cite{FO}.
In these references, only the case of a real valued function $f$ is considered. In the complex case, we give a short proof inspired by 
Lemma 10.11 p.221 of \cite{Breiman}.

\begin{proof}[Proof of Lemma \ref{lemma_weak_conv}]
Inspecting the proof of \cite{Breiman}, one can see that it is sufficient to check that $\theta \mapsto \Re\left(f(\theta)\frac{1}{1- g(i\theta)}\right)$ is integrable at $\theta=0$. The rest of the proof may be applied similarly (with minor modification) in order to prove \eqref{weak_conv}. Since
$$
\Re\left(f(\theta)\frac{1}{1- g(i\theta)}\right)=\frac{\theta^2}{|1-g(i\theta)|^2}\frac{\Re\left(f(\theta)(1- \overline{g(i\theta)})\right)}{\theta^2}
$$
and 
$$
\frac{\theta^2}{|1-g(i\theta)|^2}\longrightarrow \frac{1}{\mu^2}\quad\mbox{as }\theta\to 0,
$$
it is sufficient to prove local integrability of $\frac{\Re\left(f(\theta)(1- \overline{g(i\theta)})\right)}{\theta^2}$ at $\theta=0$.
We compute further
$$
\Re\left(f(\theta)(1- \overline{g(i\theta)})\right)= \Re\left(f(\theta)\right)\Re\left(1- \overline{g(i\theta)}\right)-\Im\left(f(\theta)\right)\Im\left(1- \overline{g(i\theta)}\right).
$$
The first term is integrable at $0$ since
$$
\int_{-\varepsilon}^\varepsilon \frac{\left|\Re\left(f(\theta)\right)\Re\left(1- \overline{g(i\theta)}\right)\right|}{\theta^2} d\theta \le  \sup_{\theta \in [-\pi,\pi]} |f(\theta)| \times \int_{-\varepsilon}^\varepsilon \frac{1-\EE(\cos(\theta X_1))}{\theta^2}d\theta <+\infty.
$$
For the integrability of the second term, we need the assumption $\Im(f(\theta))=O(\theta)$  which implies the existence of a constant $C>0$ such that $|\Im(f(\theta))|\leq C|\theta|$ for $|\theta|\leq \varepsilon$. We use also the inequality $|\sin(x)|\leq x$, $x\in\mathbb{R}$. Using this, we have
\begin{eqnarray*}
\int_{-\varepsilon}^\varepsilon \left|\frac{\Im\left(f(\theta)\right)\Im\left(1- \overline{g(i\theta)}\right)}{\theta^2}\right| d\theta
&\le & \EE \int_{-\varepsilon}^\varepsilon \left|\frac{\Im\left(f(\theta)\right)}{\theta} \frac{\sin(\theta X_1)}{\theta}\right|d\theta \\
&\le & \EE \int_{-\varepsilon}^\varepsilon C |X_1|d\theta = 2\varepsilon C \EE(|X_1|)  <+\infty .
\end{eqnarray*}
\end{proof}

\begin{proof}[Proof of  Theorem \ref{thm_expansion_v_lattice}] Let us define
\begin{eqnarray*}
u_k&:=& U(k)-U(k-1)=\sum_{n=0}^\infty \PP(S_n=k),\quad k\in \nbN,
\end{eqnarray*}
with the convention $U(-1)=0$. We use the basic fact that $S_n$ has Fourier transform $g(i\theta)^n$ and that the probabilities $\PP(S_n=k)$'s are linked to the Fourier transform by
$$
\PP(S_n=k)=\frac{1}{2\pi} \int_{-\pi}^\pi e^{-ik\theta} g(i\theta)^n d\theta,\quad n\in\nbN,\quad k\in\nbN,
$$
which can be verified by writing $g(i\theta)^n=\nbE(e^{i\theta S_n})$ and using Fubini. Hence, by Lebesgue's dominated convergence and Fubini's theorems,
\begin{eqnarray*}
u_k&=& \lim_{r\to 1^-}\sum_{n=0}^\infty r^n\PP(S_n=k)\\
&=& \lim_{r\to 1^-}\sum_{n=0}^\infty r^n \frac{1}{2\pi} \int_{-\pi}^\pi e^{-ik\theta} g(i\theta)^n d\theta\\
&=& \lim_{r\to 1^-} \frac{1}{2\pi} \int_{-\pi}^\pi e^{-ik\theta} \frac{1}{1-r g(i\theta)}d\theta.
\end{eqnarray*}
Note that $\Im(u_k)=0$. We deduce, thanks to Lemma \ref{lemma_weak_conv},
\[
u_k= \Re(u_k)=\frac{1}{2\mu}+\frac{1}{2\pi} \int_{-\pi}^\pi \Re\left(e^{-ik\theta} \frac{1}{1- g(i\theta)}\right) d\theta.
\]
We apply the same argument to the i.i.d. r.v. $(X_n')_{n\in \nbN}$ with distribution $X_n'\sim \delta_1$. This yields, for all $k\in\nbN$, 
$$ 
1= \frac{1}{2}+\frac{1}{2\pi} \int_{-\pi}^\pi \Re\left(e^{-ik\theta} \frac{1}{1- e^{i\theta}}\right)d\theta,
$$
whence we deduce
\begin{multline}
u_k-  \frac{1}{\mu}= \frac{1}{2\pi} \int_{-\pi}^\pi \Re \left[e^{-ik\theta} \left(\frac{1}{1- g(i\theta)}- \frac{1}{\mu}\frac{1}{1- e^{i\theta}}\right)\right] d\theta \\
=  \frac{1}{2\pi} \int_{-\pi}^\pi  e^{-ik\theta} \left(\frac{1}{1- g(i\theta)}- \frac{1}{\mu}\frac{1}{1- e^{i\theta}}\right)  d\theta ,\quad k\in\nbN,
\label{proof_theo_lattice_0}
\end{multline}
the last line justified by the fact that the integral is convergent. The integrand function 
$$ z\mapsto e^{-kz}\left(\frac{1}{1- g(z)}- \frac{1}{\mu}\frac{1}{1- e^{z}}\right)$$ 
is meromorphic on the domain $\{z\in\mathbb{C}; -\pi\le\Im(z)\le \pi,\ 0\le \Re(z)\le r\}$, $r<R_0$. For $R_0-r$ small enough, the poles  inside this domain are exactly $z_1,\ldots,z_N$ 
(the pole at $z_0=0$ has been removed). 
Cauchy's residue Theorem with contour given in the left panel of Figure \ref{figure_contour} implies
\begin{multline}
u_k-  \frac{1}{\mu}= -\sum_{j=1}^N \mathrm{Res}\left(\frac{e^{-kz}}{1- g(z)} ;z_j\right)\\
+ \frac{1}{2\pi} \int_{-\pi}^\pi e^{-k(r+i\theta)} \left(\frac{1}{1- g(r+i\theta)}- \frac{1}{\mu}\frac{1}{1- e^{r+i\theta}}\right)d\theta.
\label{proof_theo_lattice_1}
\end{multline}

\begin{figure}[!h]
\begin{center}
   \includegraphics[width=0.47\textwidth]{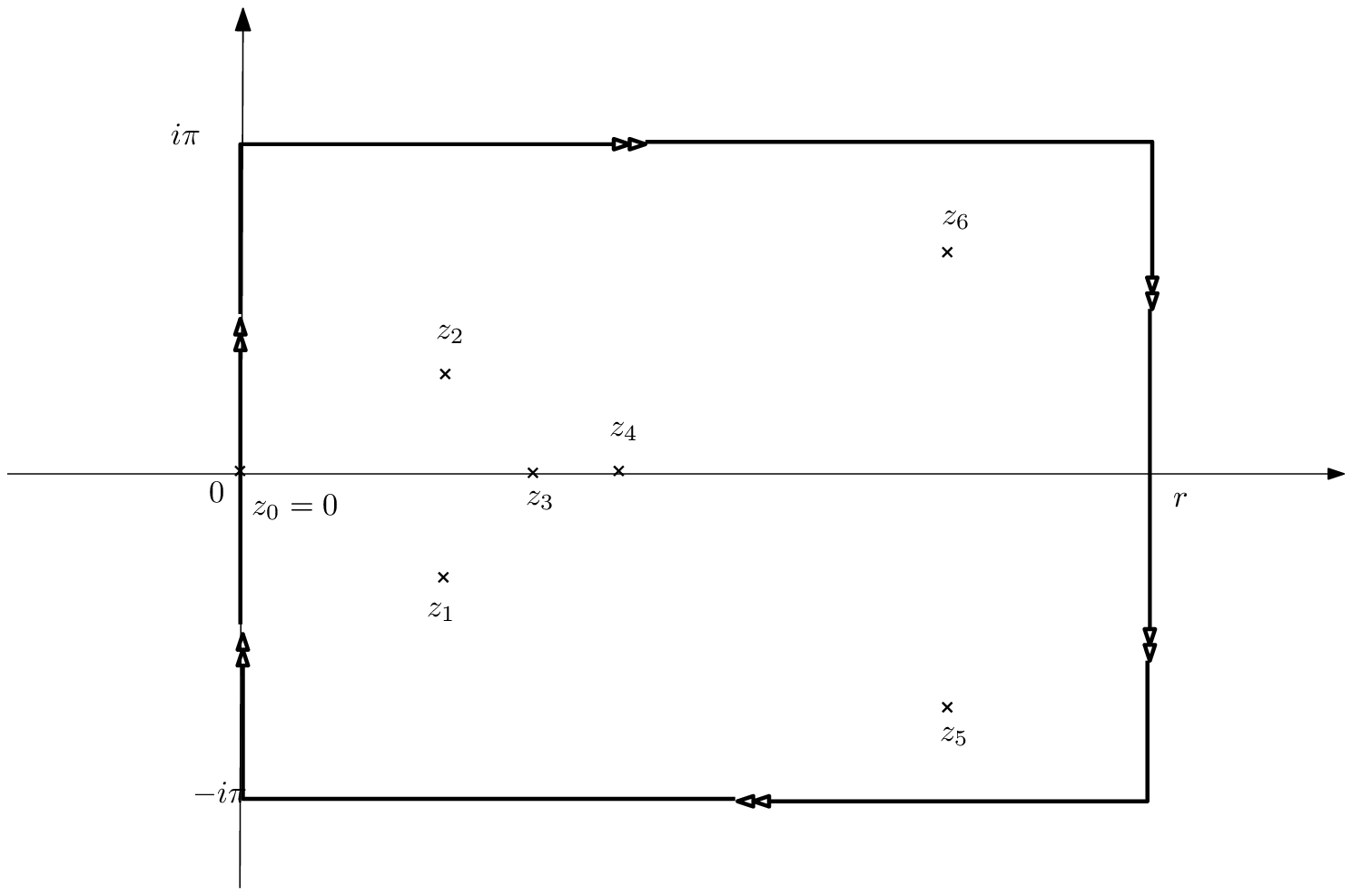} \qquad
   \includegraphics[width=0.47\textwidth]{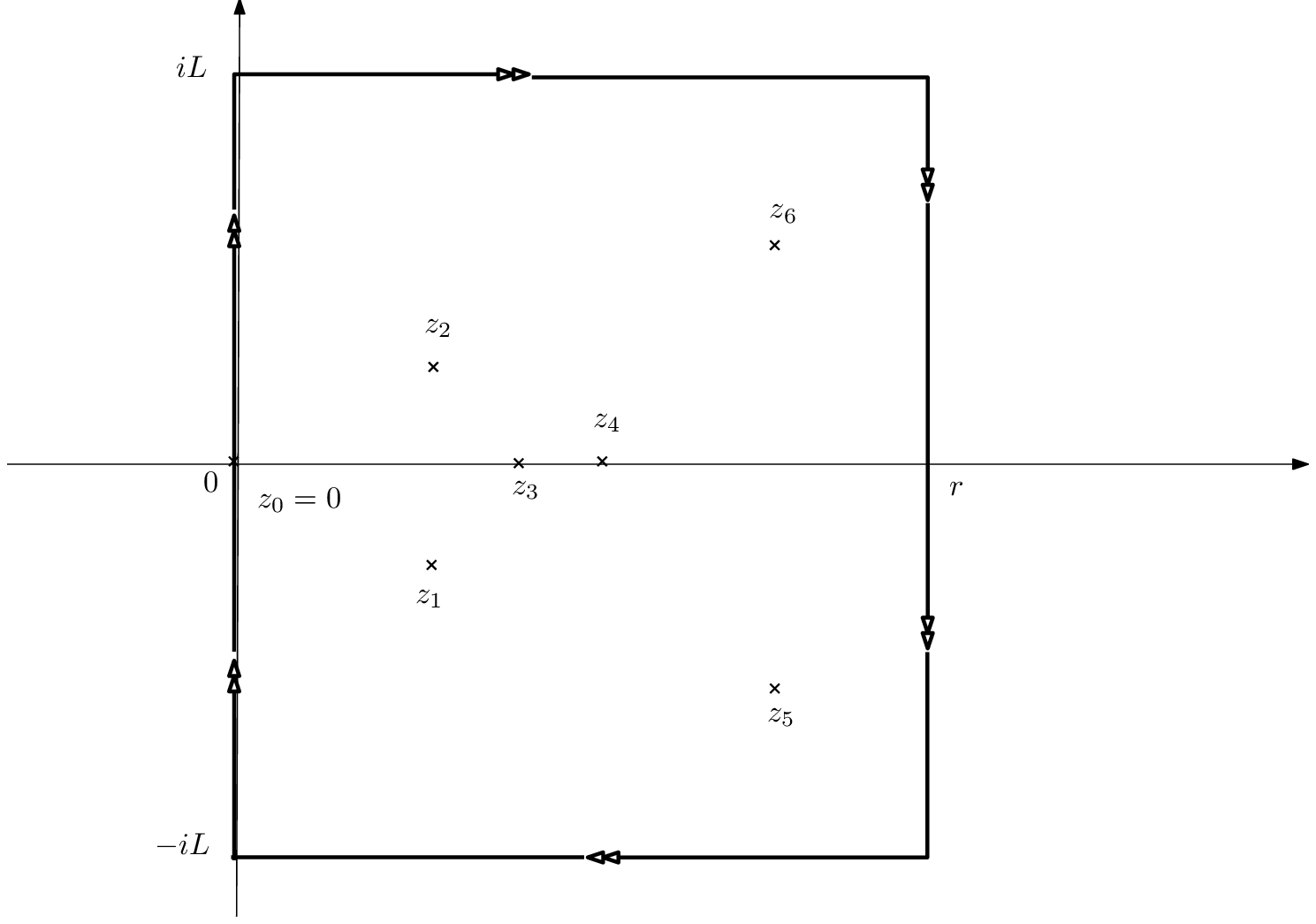}
\end{center}
\caption{Contours used for the application of Cauchy's residue Theorem in the lattice case (left) and the non-lattice case (right)}
\label{figure_contour}
\end{figure}

Using this, we obtain finally
\begin{eqnarray}
&&v(k)=\sum_{m=0}^\infty [v(k+m)-v(k+m+1)]= \sum_{m=0}^\infty [1/\mu-u(k+m+1)] \nonumber\\
&=&\sum_{m=0}^\infty \sum_{j=1}^N\mathrm{Res}\left(\frac{e^{-(k+m+1)z}}{1- g(z)} ;z_j\right)\nonumber\\
&&-
\frac{1}{2\pi}\sum_{m=0}^\infty  \int_{-\pi}^\pi e^{-(k+m+1)(r+i\theta)} \left(\frac{1}{1- g(r+i\theta)}
 - \frac{1}{\mu}\frac{1}{1- e^{r+i\theta}}\right)d\theta\nonumber \\
&=& \sum_{j=1}^N\mathrm{Res}\left(\frac{e^{-kz}}{(e^z-1)(1- g(z))} ;z_j\right)\nonumber\\
&&- \frac{e^{-kr}}{2\pi}  \int_{-\pi}^\pi \frac{e^{-ik\theta}}{e^{(r+i\theta)}-1} \left(\frac{1}{1- g(r+i\theta)}
- \frac{1}{\mu}\frac{1}{1- e^{r+i\theta}}\right)d\theta.\label{expansion_v_lattice_proof_c1}
\end{eqnarray}
Here we have used that all residues $\mathrm{Res}\left(\frac{e^{-(k+m+1)z}}{1- g(z)} ;z_j\right)$ are obtained by integrating $\frac{e^{-(k+m+1)z}}{1- g(z)}$ on a compact contour around $z_j$, so that exchanging $\sum_{m=0}^\infty$ and $\mathrm{Res}(.)$ is indeed justified by Fubini's theorem. By the Lebesgue lemma,  the last term in (\ref{expansion_v_lattice_proof_c1}) is $o(e^{-rk})$ and this proves Equation \eqref{expansion_v_lattice}. 
\end{proof}

\subsection{Proof of Proposition \ref{prop_A}}\label{section_proof_theo_expansion}
 We prove that if $F$ is absolutely continuous, then for all $R_0<R$,
\begin{equation}\label{eq:RL}
\lim_{\theta\to\pm\infty}\sup_{0\leq r\leq R_0}|g(r+i\theta)|=0.
\end{equation}
Clearly, Equation \eqref{eq:RL} implies Proposition \ref{prop_A}. 
It is worth noting that this is a uniform version of the Riemann-Lebesgue Lemma. Let us set for all $\theta\in\nbR$, $f_\theta(z):=g(z+i\theta)$, $z\in\nbC$. Since $F$ admits a density, the Riemann-Lebesgue Lemma implies that $f_\theta(z)$ converges pointwise to $0$ as $\theta\to \infty$ when $z\in\Omega:=\{z\in \nbC|\ 0\le \Re(z)\le R_0\}$. Let us now note that we have the uniform bound
$$
|f_\theta(z)|= \left| \nbE[e^{(z+i\theta)X}]\right|\le \nbE[e^{R_0 X}]<+\infty,\quad \forall z\in\Omega.
$$
By Theorem 1.6.4 p.26 of \cite{Schiff},  $f_\theta(z)$ converges uniformly towards $0$ as $\theta\to \infty$ when $z$ lies in any compact subset $K\subset \Omega$. Picking in particular $K:=\{z\in \nbC|\ \Re(z)\in [0,R_0]\mbox{ and } \Im(z)=0 \}$ yields
$$\sup_{z\in K}|f_\theta(z)|= \sup_{r\in [0,R_0]}|g(r+i\theta)|\longrightarrow 0,\quad \theta\to\infty,$$ which we were set to prove.\hfill $\Box$

\subsection{Proof of Theorem \ref{thm_expansion_v}}\label{sub_section_proof_theo_expansion}
The proof of Theorem \ref{thm_expansion_v} follows the same lines as the proof of the main Theorem in Stone \cite{Stone}.

For the sake of clarity, we divide the proof into several steps. 

\medskip\noindent
{\bf Step 1:} We prove that Condition \eqref{condA} implies that the equation $g(z)=1$ has only a finite number of solutions in $S_{R_0}$.
Condition \eqref{condA} entails the existence of  $M>0$ such that the equation $g(r+i\theta)=1$ has no solution with $0\leq r\leq R_0$ and $|\theta|>M$. 
Since furthermore the obvious bound $|g(z)|<1$ if $\Re(z)<0$ exclude solutions in the half-plane $\Re(z)<0$, 
the only possible solutions of  $g(z)=1$ in $S_{R_0}$ belong to the compact set $K=\{z\in\mathbb{C};0\leq \Re(z)\leq R_0,-M\leq\Im(z)\leq M\}$. 
The function $g(z)-1$ being holomorphic, its zeros are isolated. Hence the equation $g(z)=1$ has finitely many solution in the compact set $K$.

\medskip\noindent
{\bf Step 2:} Following Stone  \cite{Stone},  we introduce for positive $a$ and $h$,
$$
V(x,h,a):=  \nbE [U(x+h-aZ)-U(x-aZ)]=\mathbb{E}[U(x-aZ,h)]
$$
with
$$
U(x,h)=U(x+h)-U(x),
$$
and $Z$ a random variable with standard Gaussian distribution. Using Fourier analysis, one can show (similarly to Equation (7) in \cite{Stone})
\begin{eqnarray}
V(x,h,a)&=&\frac{h}{2\mu}+\frac{h}{2\pi}\int_{-\infty}^{+\infty}\Re\left\{
e^{-ix\theta}\frac{1-e^{-ih\theta}}{ih\theta}e^{-a^2\theta^2/2}\frac{1}{1-g(i\theta)}\right\}d\theta \nonumber\\
&=&\frac{h}{2\mu}+I(x,h,a). \label{V_int}
\end{eqnarray}
Setting $T(h,z):=\frac{1-e^{-hz}}{hz}$ and $\varphi(z):=\frac{1}{1-g(z)}+\frac{1}{\mu z}$, we then split the integral $I(x,h,a)$ into 
\begin{eqnarray}
I(x,h,a)&=& I_1(x,h,a)-I_2(x,h,a),\label{split_int_max_bis}\\
I_1(x,h,a)&=& \frac{h}{2\pi}\int_{-\infty}^{+\infty}\Re\left\{ e^{-ix\theta}T(h,i\theta)e^{-a^2\theta^2/2}\varphi(i\theta)\right\}d\theta, \nonumber\\
I_2(x,h,a)&=& \frac{h}{2\pi}\int_{-\infty}^{+\infty}\Re\left\{ e^{-ix\theta}T(h,i\theta)e^{-a^2\theta^2/2}  \frac{1}{i\mu\theta}\right\} d\theta.\nonumber
\end{eqnarray}
The analysis of the first term $I_1(x,h,a)$ relies on Cauchy's residue theorem. 
In the domain $S_{R_0}$, the function $\varphi(.)$ is meromorphic with poles $z_1,\ldots,z_N$ (note that the pole at $z_0=0$ has been removed). We apply Cauchy's residue theorem 
with the contour given in the right panel of Figure \ref{figure_contour} and we let the parameter $L$ tend to $+\infty$. 
Thanks to condition \eqref{condA}, the contribution of the horizontal parts of the contour vanishes as $L\to+\infty$ and we get, for all $r\in(\Re(z_N),R_0)$,
\begin{multline}\label{proof_thm_expansion_v_1}
 \frac{1}{2\pi}\int_{-\infty}^{+\infty} e^{-ix\theta}T(h,i\theta)e^{-a^2\theta^2/2}\varphi(i\theta)d\theta
\\
= -\sum_{j=1}^N \mathrm{Res}\left(e^{-xz}T(h,z)e^{a^2z^2/2}\varphi(z);z_j \right)+ \Re \tilde{I}_1(x,h,a,r)
\end{multline}
with
\begin{equation}\label{def_I_tilde_1}
\tilde{I}_1(x,h,a,r)= \frac{1}{2\pi}\int_{-\infty}^{+\infty} e^{-x(r+i\theta)}T(h,r+i\theta)e^{a^2(r+i\theta)^2/2}\varphi(r+i\theta)d\theta.
\end{equation}
Note that the sum of residues is a real number because for conjugate poles $z$ and $\bar z$, the residues at $z$ and $\bar z$ are also conjugate so that the sum is real.
Similarly for the second term $I_2(x,h,a)$, Cauchy's residue theorem yields
\begin{equation}\label{proof_thm_expansion_v_1bis} \frac{1}{2\pi}\int_{-\infty}^{+\infty} e^{-ix\theta}T(h,i\theta)e^{-a^2\theta^2/2}\frac{1}{i\mu\theta} d\theta
= - \frac{1}{2\mu}+  \tilde{I}_2(x,h,a,r)
\end{equation}
with
\begin{equation}\label{def_I_tilde_2}
\tilde{I}_2(x,h,a,r)= \frac{1}{2\pi}\int_{-\infty}^{+\infty} e^{-x(r+i\theta)}T(h,r+i\theta)e^{a^2(r+i\theta)^2/2}\frac{1}{\mu (r+i\theta)}d\theta.
\end{equation}
Equations \eqref{V_int}, \eqref{split_int_max_bis}, \eqref{proof_thm_expansion_v_1} and \eqref{proof_thm_expansion_v_1bis} together yield
\begin{multline}\label{proof_thm_expansion_v_2_0}
V(x,h,a)=\frac{h}{\mu}-h\sum_{j=1}^N \mathrm{Res}\left(e^{-xz}T(h,z)e^{a^2z^2/2}\varphi(z);z_j \right)\\
+h\Re\{\tilde{I}_1(x,h,a,r)\}-h\Re\{\tilde{I}_2(x,h,a,r)\}. 
\end{multline}
Note that the Rieman-Lebesgue Lemma implies that the second term $\widetilde I_2(x,h,a,r)$ satisfies, for all $r>0$,
\begin{equation}
\tilde{I}_2(x,h,a,r)= o(e^{-rx}) \quad \mbox{uniformly for $a$ and $h$ in bounded sets}. 
\label{estimate_J_2}
\end{equation}

\medskip \noindent
{\bf Step 3:} We analyze here the term $\tilde{I}_1(x,h,a,r)$ and, similarly as Equation (8) from \cite{Stone}, we show that 
\begin{equation}\label{proof_thm_expansion_v_8}
\tilde{I}_1(x,h,a,r)=o(e^{-rx}(1+|\ln(a)|))\quad \mbox{uniformly for $a$ and $h$ in bounded sets}.
\end{equation}
The argument for this statement is given almost without proof in \cite{Stone} but can be
adapted from a similar argument  in \cite{Stone2} in the
following way. One notices that
$$
\varphi(z)=\frac{g(z)-1-\mu z}{\mu^2z^2} + \frac{(g(z)-1-\mu z)^2}{\mu^2 z^2(1-g(z))},
$$
so that we have, for $a<1$,
\begin{equation}
\tilde{I}_1(x,h,a,r)= J_1 + J_2+J_3+J_4,\label{proof_thm_expansion_v_9} 
\end{equation}
with
\begin{eqnarray}
J_1&=&\frac{1}{2\pi}\int_{-\infty}^{+\infty} e^{-x(r+i\theta)}T(h,r+i\theta)
e^{a^2(r+i\theta)^2/2}  \frac{g(r+i\theta)-1-\mu (r+i\theta)}{\mu^2
(r+i\theta)^2} d\theta,\nonumber\\
J_2&=& \frac{1}{2\pi}\int_{|\theta|\le 1} e^{-x(r+i\theta)}T(h,r+i\theta)
e^{a^2(r+i\theta)^2/2} \frac{(g(r+i\theta)-1-\mu
(r+i\theta))^2}{\mu^2 (r+i\theta)^2(1-g(r+i\theta))} d\theta.
\nonumber\\
J_3&=&  \frac{1}{2\pi}\int_{1< |\theta|< a^{-2}} e^{-x(r+i\theta)}T(h,r+i\theta)
e^{a^2(r+i\theta)^2/2} \frac{(g(r+i\theta)-1-\mu
(r+i\theta))^2}{\mu^2 (r+i\theta)^2(1-g(r+i\theta))} d\theta.
\nonumber\\ 
J_4&=&  \frac{1}{2\pi}\int_{|\theta|\ge a^{-2}} e^{-x(r+i\theta)}T(h,r+i\theta)
e^{a^2(r+i\theta)^2/2} \frac{(g(r+i\theta)-1-\mu
(r+i\theta))^2}{\mu^2 (r+i\theta)^2(1-g(r+i\theta))} d\theta.
\nonumber 
\end{eqnarray}
Since $\frac{g(z)-1-\mu z}{\mu^2 z^2}$ is analytic on $S_R$ (the
pole at $z=0$ has been removed), Cauchy's residue theorem and Lebesgue's
Lemma yield 
\[
J_1=o(e^{-rx})\quad \mbox{uniformly for $a$ and $h$ in bounded sets}.
\]
For the other terms, we use the fact that the function
$$F(\theta):=T(h,r+i\theta)
e^{a^2(r+i\theta)^2/2} \frac{(g(r+i\theta)-1-\mu
(r+i\theta))^2}{\mu^2 (r+i\theta)^2(1-g(r+i\theta))}$$ 
is bounded for $|\theta|\le 1$ and satisfies
$$F'(\theta)=O\left(\frac{1}{\theta}\right)\quad \mbox{as } |\theta|\ge 1$$
and
$$ F(\theta)=O\left( e^{-a^2 \theta^2/2}\right)\quad \mbox{as } |\theta|\ge 1.$$
This justifies the following estimates
\begin{eqnarray*}
|J_2| & =  &\frac{1}{2\pi} \left|\int_{|\theta|\le 1} e^{-x(r+i\theta)} F(\theta) d\theta \right| =O(e^{-rx}),\\
|J_3| & =  &\frac{1}{2\pi} \left|\int_{1< |\theta|< a^{-2}} e^{-x(r+i\theta)} F(\theta) d\theta \right|= O(e^{-rx}|\log(a)|),\\
|J_4| & =  &\frac{1}{2\pi} \left|\int_{ |\theta|\ge a^{-2}} e^{-x(r+i\theta)} F(\theta) d\theta \right|=O(e^{-rx}).
\end{eqnarray*}
Gathering the above inequalities in (\ref{proof_thm_expansion_v_9}), we  obtain (\ref{proof_thm_expansion_v_8}). The term $O(e^{-rx}|\log(a)|)$ can be replaced by
$o(e^{-rx}|\log(a)|)$ because $r\in (\Re(z_N),R_0)$ is arbitrary.
Equations  \eqref{proof_thm_expansion_v_2_0}, \eqref{estimate_J_2} and \eqref{proof_thm_expansion_v_8} together yield
\begin{equation}\label{proof_thm_expansion_v_2}
V(x,h,a)-\frac{h}{\mu}=-h\sum_{j=1}^N \mathrm{Res}\left(e^{-xz}T(h,z)e^{a^2z^2/2}\varphi(z);z_j \right)+o(e^{-rx}(1+|\ln(a)|)).
\end{equation}

\noindent
\medskip
{\bf Step 4:} We prove the following inequality: 
\begin{multline}
V(x+e^{-rx},1-2e^{-rx},e^{-rx}/x)-Me^{-rx}\le U(x,1)\\
\le (1-e^{-rx})^{-1}V(x-e^{-rx},1+2e^{-rx},e^{-rx}/x).
\label{proof_thm_expansion_v_3}
\end{multline} 
Recall the definition  $V(x,h,a)=\nbE[U(x-aZ,h)]$ with $U(x,h)=U(x+h)-U(x)$ and 
$Z$ a standard Gaussian random variable. Using the fact that $U(.)$ is non decreasing, we check, for $|y|\le e^{-rx}$,
$$
U(x+e^{-rx}-y, 1-2e^{-rx})\le U(x,1)\le U(x-e^{-rx}-y, 1+2e^{-rx}).
$$
Let $x_0>0$ be such that, for all $x\ge x_0$, $\nbP(|Z|\ge x)\le e^{-rx}\leq 1/2$. 
Besides, the fact that renewal function $ x\mapsto U(x)$ is sub-additive implies that there exists some constant $M$ independent from $h$ and $x$ such that 
\begin{equation}\label{eq:bound}
0\le U(x,h)\le M,\quad \mbox{for all $x>0$ and $h$ in a bounded set}.
\end{equation}
By the definition $V(x,h,a)=\nbE[U(x-aZ,h)]$, both sides of Equation \eqref{proof_thm_expansion_v_3}
are obtained by splitting 
$$V(x\pm e^{-rx},h,a)=\nbE[U(x\pm e^{-rx}-aZ,h)\nbu_{\{ |Z|\le x\}}]+\nbE[U(x\pm e^{-rx}-aZ,h)\nbu_{\{ |Z|> x\}}]$$
with the values $a=a(x)=e^{-rx}/x$ and $h=h(x)=1\pm 2e^{-rx}$. 

\noindent
\medskip
{\bf Step 5:} Let us now prove the following estimate
\begin{equation}
U(x,1)-\frac{1}{\mu}=-\sum_{j=1}^N \mathrm{Res}\left(e^{-xz}T(1,z)\varphi(z);z_j \right)+o(e^{-rx}).
\label{proof_thm_expansion_v_4}
\end{equation}
This is roughly obtained from \eqref{proof_thm_expansion_v_2} by setting 
\[
a=a(x)=e^{-rx}/x\quad \mbox{and}\quad h=h_\pm(x)=1\pm 2e^{-rx}.
\] 
This can be made rigorous thanks to Equation \eqref{proof_thm_expansion_v_3}. We have, uniformly in the neighborhood of $z_j$, $j=1,\ldots,N$,
\begin{eqnarray*}
e^{-(x\pm  e^{-rx} )z}T(h(x),z)e^{-a(x)^2 z^2/2}&=&e^{-xz}[1+o(e^{-rx})]\\
&& .\left[T(1,z)+ o(e^{-rx})\right][1+o(e^{-rx})]\\
&=&e^{-xz} T(1,z)+ o(e^{-rx})
\end{eqnarray*}
so that
\begin{equation}
\mathrm{Res}\left(e^{-(x\pm  e^{-rx} )z}T(h,z)e^{a^2z^2/2}\varphi(z);z_j \right)
=\mathrm{Res}\left(e^{-xz}T(1,z)\varphi(z);z_j \right) + o(e^{-rx}).\label{proof_thm_expansion_v_5}
\end{equation}
Besides, left hand side of Equation \eqref{proof_thm_expansion_v_3} and Equation \eqref{proof_thm_expansion_v_5} entail
\begin{eqnarray}
&&U(x,1)-\frac{1}{\mu}+\sum_{j=1}^N \mathrm{Res}\left(e^{-xz}T(1,z)\varphi(z);z_j \right)\nonumber\\
&\ge& V(x+e^{-rx},1-2e^{-rx},e^{-rx}/x)-Me^{-rx} -\frac{1}{\mu}\nonumber\\
&& +\sum_{j=1}^N \mathrm{Res}\left(e^{-xz}T(1,z)\varphi(z);z_j \right)\nonumber\\
&=& V(x+e^{-rx},h(x),a(x))-Me^{-rx} \nonumber\\
&&-\frac{1}{\mu}+\sum_{j=1}^N \mathrm{Res}\left(e^{-(x\pm  e^{-rx} )z}T(h,z)e^{a^2z^2/2}\varphi(z);z_j \right)
 +o(e^{-rx}) \label{proof_thm_expansion_v_6}
\end{eqnarray}
Now from \eqref{proof_thm_expansion_v_2}, we have
\begin{multline*}
V(x+e^{-rx},h(x),a(x)) -\frac{h(x)}{\mu} +h(x)\sum_{j=1}^N \mathrm{Res}\left(e^{-(x\pm  e^{-rx} )z}T(h,z)e^{a^2z^2/2}\varphi(z);z_j \right) \\=o(xe^{-rx}).
\end{multline*}
Together with \eqref{proof_thm_expansion_v_6} and since $h(x)=1+O(e^{-rx})$ , this yields
\begin{equation}
U(x,1)-\frac{1}{\mu}+\sum_{j=1}^N \mathrm{Res}\left(e^{-xz}T(1,z)\varphi(z);z_j \right)
\ge o(xe^{-rx}),\label{proof_thm_expansion_v_7}.
\end{equation}
A similar upper bound is proved in the same way, using the right hand side of Equation \eqref{proof_thm_expansion_v_3}.
We deduce
\[
U(x,1)-\frac{1}{\mu}+\sum_{j=1}^N \mathrm{Res}\left(e^{-xz}T(1,z)\varphi(z);z_j \right)=o(xe^{-rx}).
\]
We can replace $o(xe^{-rx})$ by $o(e^{-rx})$ because $r$ is arbitrary in  $(\Re(z_N),R_0)$. This proves Eq. \eqref{proof_thm_expansion_v_4}.

\noindent
\medskip
{\bf Step 6:} We finally prove Equation \eqref{expansion_v}. Since $\lim_{x\to \infty}v(x)=0$, we have
\begin{equation}\label{eq:step6_1}
 v(x)= \sum_{k=0}^\infty [v(x+k)-v(x+k+1)]= \sum_{k=0}^\infty [-U(x+k,1)+1/\mu].
\end{equation}
Using Equation \eqref{proof_thm_expansion_v_4}, we deduce, for all $ r\in(\Re(z_N),R_0)$,
\begin{eqnarray}
v(x)&=& \sum_{k=0}^\infty \sum_{j=1}^N \mathrm{Res}\left(e^{-(x+k)z}T(1,z)\varphi(z);z_j \right)+o(e^{-r(x+k)})\nonumber\\
&=& \sum_{j=1}^N \mathrm{Res}\left(\frac{e^{-xz}}{1-e^{-z}}T(1,z)\varphi(z);z_j \right)+o(e^{-rx})\nonumber\\
&=& \sum_{j=1}^N \mathrm{Res}\left(\frac{e^{-xz}}{z(1-g(z))};z_j \right)+o(e^{-rx})\label{eq:step6_2}.
\end{eqnarray}
Equation \eqref{expansion_v} follows easily.\hfill $\Box$

\subsection{Proof of Corollaries \ref{cor_expansion_exact} and  \ref{cor_expansion_exact_2}}
\begin{proof}[Proof of Corollary \ref{cor_expansion_exact}]
We consider first the lattice case. According to Condition \eqref{eq:cond_lattice}, one can consider $(r_n)_{n\geq 1}$ an increasing sequence such that $r_n\to +\infty$ and 
\[
\lim_{n\to+\infty} \sup_{-\pi\leq \theta\leq \pi} \left|\frac{1}{e^{r_n} (1-g(r_n+i\theta))} \right|=0.
\]
According to Equation \eqref{expansion_v_lattice_proof_c1},
\begin{multline*}
v(k)= \sum_{j=1}^{N(r_n)}\mathrm{Res}\left(\frac{e^{-kz}}{(e^z-1)(1- g(z))} ;z_j\right)\\
-\frac{1}{2\pi}  \int_{-\pi}^\pi \frac{e^{-k(r_n+i\theta)}}{e^{r_n+i\theta}-1} \left(\frac{1}{1- g(r_n+i\theta)}- \frac{1}{\mu}\frac{1}{1- e^{r_n+i\theta}}\right)d\theta
\end{multline*}
with $N(r_n)$ the number of solutions of the equation $g(z)=1$ in $S_{r_n}^f$.
Lebesgue's dominated convergence Theorem implies, for $k\geq 1$, 
\[
\lim_{n\to+\infty}\int_{-\pi}^\pi \frac{e^{-k(r_n+i\theta)}}{e^{r_n+i\theta}-1} \left(\frac{1}{1- g(r_n+i\theta)}- \frac{1}{\mu}\frac{1}{1- e^{r_n+i\theta}}\right)d\theta= 0.
\]
This yields the result (\ref{exact_expansion_v_lattice}).

We now consider the non-lattice case. Thanks to Assumption (\ref{eq:cond_nonlattice}), let $(r_n)_{n\geq 1}$ be such that $r_n\to +\infty$ and
$$\sup_{\theta\in\mathbb{R}}\left|\frac{1}{ r_n(1-g(r_n+i\theta))}\right| \longrightarrow 0,\quad n\to\infty.$$
We will prove below that setting $h=1$, $r=r_n$ and $a=a_n(x)=e^{-r_n x}/x$ in  Equation \eqref{proof_thm_expansion_v_2_0} and letting $n\to +\infty$, we obtain
\begin{equation}\label{eq:cor}
U(x,1)=\frac{1}{\mu}-\sum_{j=1}^\infty \mathrm{Res}\left(e^{-xz}T(1,z)\varphi(z);z_j \right) = \frac{1}{\mu}-\sum_{j=1}^\infty \frac{e^{-xz_j}(e^{-z_j}-1)}{z_j g'(z_j)},
\end{equation}
remembering that roots $(z_j)_{j\in\nbN}$ are simple. This is justified as follows:
\begin{itemize}[leftmargin=*]
\item[-] Condition \eqref{eq:bound} together with Lebesgue's dominated convergence Theorem implies
\begin{equation}\label{conv_V_h_a}
V(x,h,a)=V(x,1,e^{-r_n x}/x)\to U(x,h)\quad \mbox{as}\ n\to +\infty.
\end{equation} 
\item[-] Recalling that $\tilde I_1(x,h,a,r)$ and $\tilde I_2(x,h,a,r)$ are respectively given by (\ref{def_I_tilde_1}) and (\ref{def_I_tilde_2}), the integral term $\tilde I_1(x,h,a,r)-\tilde I_2(x,h,a,r)$ with $h=1$, $r=r_n$ and $a=e^{-r_n x}/x$ converge to $0$ as $n\to+\infty$. Indeed, condition \eqref{eq:cond_lattice} and the upper bound $|T(h,r+i\theta)|\leq 2/r$ imply
\begin{eqnarray*}
&&|\tilde I_1(x,h,a,r)-\tilde I_2(x,h,a,r)|\\
&= & \frac{1}{2\pi}\left| \int_{-\infty}^{+\infty} e^{-x(r+i\theta)}T(h,r+i\theta)e^{a^2(r+i\theta)^2/2}\frac{1}{1-g(r+i\theta)}d\theta\right| \\
&\leq &\frac{e^{-xr}}{\pi} \int_{-\infty}^{+\infty} \left| \frac{1}{r(1-g(r+i\theta))}\right|e^{a^2r^2/2-a^2\theta^2/2}d\theta \\
&\leq & \sup_{\theta\in\mathbb{R}}\left| \frac{1}{r(1-g(r+i\theta))}\right| \times \frac{e^{-xr}}{\pi}\frac{e^{a^2r^2/2}\sqrt{2\pi}}{a}\\
&\leq& \sup_{\theta\in\mathbb{R}}\left| \frac{1}{r_n(1-g(r_n+i\theta))}\right| \times \frac{\sqrt{2}x}{\sqrt{\pi}}\exp\left( \frac{r_n^2e^{-2r_nx}}{2x}\right).
\end{eqnarray*}  
This last quantity goes to $0$ as $r_n\to +\infty$.
\item[-] Let us recall inequality $|e^z-1|\le e|z|$ for all $|z|\le 1$. As roots $(z_n)_{n\in\nbN}$ are simple and real, one has for all $j=1,...,N(r_n)$, and for $n$ large enough,
\begin{multline*}
\left|\mathrm{Res}\left(e^{-(x\pm  e^{-r_nx} )z}T(1,z)e^{a_n(x)^2z^2/2}\varphi(z);z_j\right)
- \mathrm{Res}\left(e^{-xz}T(1,z) \varphi(z);z_j\right)\right|\\
=\left|\mathrm{Res}\left(\left[ e^{\pm  e^{-r_nx} z}e^{a_n(x)^2z^2/2}-1\right] e^{-xz}T(1,z) \varphi(z);z_j\right)\right|\\
=\left| \left[ e^{\pm  e^{-r_nx} z_j}e^{a_n(x)^2z_j^2/2}-1\right]\right|.\left|\mathrm{Res}\left( e^{-xz}T(1,z) \varphi(z);z_j\right)\right|\\
\le  e\left| \pm e^{-r_nx}z_j +a_n(x)^2z_j^2/2 \right|.\left|\mathrm{Res}\left( e^{-xz}T(1,z) \varphi(z);z_j\right)\right|\\
\le M e^{-2 r_nx}r_n^2 .\left|\mathrm{Res}\left( e^{-xz}T(1,z) \varphi(z);z_j\right)\right|
\end{multline*}
where $M=M(x)$ is a constant independent from $n$ and $j$, but which may depend on $x>0$. One thus deduces
\begin{multline}
\sum_{j=1}^{N(r_n)}\left|  \mathrm{Res}\left(e^{-(x\pm  e^{-r_nx} )z}T(1,z)e^{a_n(x)^2z^2/2}\varphi(z);z_j\right)\right.\\
\left. - \mathrm{Res}\left(e^{-x z}T(1,z) \varphi(z);z_j\right)\right|\le  M e^{-2 r_nx}r_n^2 \sum_{j=1}^{N(r_n)} \left| \mathrm{Res}\left(e^{-x z}T(1,z) \varphi(z);z_j\right)\right|
\label{inequality_infinite_sum}
\end{multline}
which tends to zero as $n\to\infty$ because of Assumption (\ref{exact_expansion_abs_convergence}).
\item[-] The fact that $\tilde I_1(x,h,a_n(x),r_n)-\tilde I_2(x,h,a_n(x),r_n)$ tends to $0$ as $n\to\infty$ and convergence (\ref{conv_V_h_a}) (with $h=1$) as well as inequality (\ref{inequality_infinite_sum}) implies \eqref{eq:cor} from (\ref{proof_thm_expansion_v_2_0}).
\end{itemize}
The end of the proof follows easily, as in Step 6 of the proof of Theorem \ref{thm_expansion_v}: Equations \eqref{eq:step6_1} and \eqref{eq:cor} imply (\ref{exact_expansion_v_non_lattice}).
\end{proof}

\begin{proof}[Proof of Corollary \ref{cor_expansion_exact_2}]
In the lattice case, we deduce from Corollary~\ref{cor_expansion_exact} that
\begin{eqnarray*}
u(k)&=&U(k)-U(k-1)=v(k)-v(k-1)+\frac{1}{\mu}\\
&=& \frac{1}{\mu}+\sum_{j=1}^N \mathrm{Res}\left(\frac{e^{-kz}-e^{-(k-1)z}}{(e^z-1)(1-g(z))};z_j \right)\\
&=& \frac{1}{\mu}-\sum_{j=1}^N \mathrm{Res}\left(\frac{e^{-kz}}{1-g(z)};z_j \right).
\end{eqnarray*}
Finally, note that $z_0=0$, $g(0)=1$, $g'(0)=\mu$ so that
\[
\frac{1}{\mu}=\mathrm{Res}\left(\frac{e^{-kz}}{1-g(z)};z_0 \right).
\]
In the non-lattice case, we deduce from  Corollary \ref{cor_expansion_exact} that
\begin{eqnarray*}
u(x)&=&\frac{\mathrm{d}U(x)}{\mathrm{d}x} =\frac{1}{\mu}+\frac{\mathrm{d}v(x)}{\mathrm{d}x}\\
&=& \frac{1}{\mu}-\sum_{j=1}^N \mathrm{Res}\left(\frac{e^{-kz}}{1-g(z)};z_j \right).
\end{eqnarray*}
The result follows since $\frac{1}{\mu}=\mathrm{Res}\left(\frac{e^{-kz}}{1-g(z)};z_0 \right)$.
\end{proof}

\section{Proofs for section \ref{sec:ruin}}

\subsection{Proof of Theorem \ref{prop_exp_ruin}}
Assumption \eqref{condB} that the moment generating function of $G$ is finite everywhere is equivalent to the fact that the tail function $\bar G=1-G$ decreases super-exponentially fast, i.e.  $\bar G$ is in $\mathcal{D}$ with
\begin{equation}
{\cal D}:=\left\{ f:[0,+\infty)\longrightarrow \nbC\left|\ \forall r>0,\quad f(x)=o(e^{-rx})\mbox{ as } x\to \infty\right.\right\}.
\label{def_cal_D}
\end{equation}
The following lemma can be verified easily.
\begin{lemm}\label{lemme_D}
Let $f\in {\cal D}$, then
\begin{enumerate}
	\item $x\mapsto e^{zx}f(x)$ belongs to ${\cal D}$ for all $z\in\nbC$;
	\item $x\mapsto \displaystyle \int_x^\infty f(y)dy$ belongs to ${\cal D}$.
\end{enumerate}
\end{lemm}

\begin{proof}[Proof of Theorem \ref{prop_exp_ruin}]
It is well known  that $\psi(x)$ satisfies the following defective renewal equation
\begin{equation}\label{renewal_eq}
\psi(x)= \bar{L}(x)+\int_0^x \psi(x-y)dL(y),
\end{equation}
where
\begin{equation}
L(x):=\int_0^x \frac{\alpha}{c} \bar{G}(y)dy,\quad x\ge 0,
\label{expression_L(x)}
\end{equation}
can be seen as the cdf of a defective distribution (because $L(+\infty)=\frac{\alpha}{c} m<1$), 
see e.g. Corollary 3.3 p.79 of Asmussen and Albrecher \cite{AA} or Eq. (7.2) p.377 of Feller \cite{Feller}.

We introduce the probability distribution $dF(x)=e^{\kappa x}dL(x)$ and use the notation
\begin{multline}
g(z):=\int_0^\infty e^{zx}dF(x)=\frac{\alpha}{c} \frac{1}{z+\kappa}\EE \left( e^{(z+\kappa) Z}-1\right),\\
\mu:=\int_0^\infty xdF(x) \quad
\mbox{and}\quad \mu_2:=\int_0^\infty x^2dF(x)
\label{def_g_ruin}
\end{multline}
for the moment generating function and the first and second moments of $F$. Elementary calculation yields
\begin{equation}
g'(z_j)=\frac{1}{z_j+\kappa}\left[\frac{\alpha}{c} \EE \left( Ze^{(z_j+\kappa)Z}\right)-1\right].
\label{g(zj)_prime_ruin}
\end{equation}
Thanks to the assumption $\bar G\in\mathcal{D}$ and Lemma \ref{lemme_D}, the tail function of $F$ belongs also to $\mathcal{D}$ whence the moment generating function $g(z)$ is defined for all $z\in\nbC$. 
Setting  $\psi_\kappa(x):=e^{\kappa x}\psi(x)$, \eqref{renewal_eq} entails the non-defective  renewal equation
\begin{equation}\label{renewal_eq_dieze}
\psi_\kappa(x)= z(x)+\int_0^x \psi_\kappa(x-y)dF(y),\quad z(x):= e^{\kappa x}\bar{L}(x),
\end{equation}
see Feller \cite{Feller} Eq. (6.12) p.376. The distribution $F$ is non-lattice so that Smith's renewal theorem entails the classical Cramer-Lundberg asymptotics:
\begin{equation}
\psi_\kappa(x) \to C:=\frac{\int_0^\infty z(y) dy}{\int_0^\infty ydF(y)} =\frac{c-\alpha m}{c-\alpha \EE\left( Z_1 e^{\kappa Z_1}\right)}\quad \mbox{as } x\to+\infty,
\label{constant_C}
\end{equation}
the last equality being obtained by direct computation, see e.g.  Theorem 5.3 p.86 of \cite{AA}. 

We now wish to provide an expansion of $\psi_\kappa(x)$ and provide extra terms in (\ref{constant_C}). Let us introduce
$$
U(x)=\sum_{n=0}^\infty F^{\ast (n)}(x),\quad v(x)=U(x)-\frac{x}{\mu}-\frac{\mu_2}{2\mu^2},\quad x\ge 0.
$$
The solution to (\ref{renewal_eq_dieze}) is given by
\begin{equation}
\psi_\kappa(x)= \int_0^x z(x-y) dU(y) = \int_0^x z(x-y) \frac{dy}{\mu}+\int_0^x z(x-y) dv(y).\label{renewal_explicit}
\end{equation}
Using the fact that $\int_0^x z(x-y)dy=\int_0^x z(y)dy$, an integration by parts and the definition of $C$, we obtain 
\begin{eqnarray}
 \psi_\kappa(x)-C&=& -\int_x^\infty z(y)\frac{dy}{\mu} + \int_0^x z'(x-y) v(y)dy- v(0-)z(x) + \frac{\alpha}{c} m v(x)\nonumber\\
 &:=& I_1(x) + I_2(x) + I_3 (x)+I_4(x).
\label{expansion_ruin1}
\end{eqnarray}
We recall that $x\mapsto \bar{G}(x)$ belongs to $\cal D$ defined in (\ref{def_cal_D}), so that one easily sees that $z(x)$ defined in (\ref{renewal_eq_dieze}) belongs to $\cal D$, and that $I_1(x)$ and $I_3(x)$ are $ o(e^{-rx})$ from Lemma \ref{lemme_D}. Thus we only need to study asymptotics of the terms $I_2(x)$ and $I_4(x)$ as $x\to \infty$ in (\ref{expansion_ruin1}).

We start with $I_4(x)$. Let us note that, as $g(z)$ is defined for all $z\in\nbC$, $R$ defined in (\ref{def_R}) is equal to $+\infty$, as well as $R_0$ in Theorem \ref{thm_expansion_v}. Expansion (\ref{expansion_v}) in Theorem \ref{thm_expansion_v} yields
\begin{equation}
I_4(x)=\sum_{j=1}^N \frac{\alpha}{c} m \mathrm{Res}\left(\frac{e^{-xz}}{z(1-g(z))};z_j\right) + o(e^{-rx})
\label{exp_I4_1}
\end{equation}
Now turning to $I_2(x)$. By a change of variable one has that $I_2(x)=\int_0^x z'(y) v(x-y)dy$, so that, using Expansion (\ref{expansion_v}) in Theorem \ref{thm_expansion_v} for $v(x-y)$ yields
\begin{equation}
I_2(x)= \sum_{j=1}^N \int_0^x z'(y)\mathrm{Res}\left(\frac{e^{-(x-y)z}}{z(1-g(z))};z_j\right)dy+\int_0^x z'(y) e^{-r(x-y)}\varepsilon(x-y)dy 
\label{exp_I2_1}
\end{equation}
for some function $\varepsilon(.)$ vanishing at $+\infty$. Furthermore, it is easy to check that $x\mapsto z'(x) \in {\cal D}$. One then verifies that $\int_x^\infty z'(y)\mathrm{Res}\left(\frac{e^{-(x-y)z}}{z(1-g(z))};z_j\right)dy$ is an $o(e^{-rx})$, by writing the residue as an integral over a fixed contour including $z_j$ then using Lemma \ref{lemme_D}. One also checks that the last term on the right hand side of (\ref{exp_I2_1}) is $o(e^{-rx})$ by a dominated convergence argument. One may thus write in (\ref{exp_I2_1})
\begin{equation}
I_2(x)= \sum_{j=1}^N \int_0^\infty z'(y)\mathrm{Res}\left(\frac{e^{-(x-y)z}}{z(1-g(z))};z_j\right)dy+o(e^{-rx}).
\label{exp_I2_1_bis}
\end{equation}
Deriving $z(y)=e^{\kappa y}\bar{L}(y)$ with  $L(y)$ given in (\ref{expression_L(x)}), expressing the residue as an integral over a fixed contour and using Fubini, thus yields the following expression for the terms in the summation on the right hand side of (\ref{exp_I2_1_bis}) :
\begin{eqnarray*}
&& \int_0^\infty z'(y)\mathrm{Res}\left(\frac{e^{-(x-y)z}}{z(1-g(z))};z_j\right)dy\nonumber\\
&& =\mathrm{Res}\left(\frac{e^{-x z}}{z(1-g(z))} \int_0^\infty z'(y) e^{y z}dy;z_j\right)\nonumber\\
&& =\mathrm{Res}\left(\frac{e^{-x z}}{z(1-g(z))} \left[\int_0^\infty \kappa e^{(z+\kappa)y}\int_y^\infty \frac{\alpha}{c} \bar{G}(v)dv dy \right.\right. \nonumber\\
&& - \left.\left.\int_0^\infty e^{(z+\kappa)y} \frac{\alpha}{c} \bar{G}(y)dy\right];z_j\right)\nonumber\\
&&=  \mathrm{Res}\left(\frac{e^{-x z}}{z(1-g(z))} \left[ \frac{\alpha}{c} \EE\left( \int_0^\infty \int_0^\infty \kappa e^{(z+\kappa)y} \nbu_{[Z\ge v\ge y]} dy dv \right)\right. \right.\nonumber\\
&& - \left. \left.\frac{\alpha}{c} \EE\left( \int_0^\infty e^{(z+\kappa)y} \nbu_{[Z\ge y]}dy\right)\right];z_j\right)\nonumber\\
&&= \mathrm{Res}\left(\frac{e^{-x z}}{z(1-g(z))} \left[ \frac{\alpha}{c} \frac{\kappa}{z+\kappa}\EE\left( \frac{e^{(z+\kappa)Z}-1}{z+\kappa}-Z\right)\right.\right.\nonumber\\
&& - \left. \left.\frac{\alpha}{c}\frac{1}{z+\kappa}\EE\left( e^{(z+\kappa)Z}-1\right)\right];z_j\right),
\end{eqnarray*}
which, since $z_j$ satisfies $g(z_j)=1$ with $g(z)$ given by (\ref{def_g_ruin}), yields
\begin{eqnarray}
&&\int_0^\infty z'(y)\mathrm{Res}\left(\frac{e^{-(x-y)z}}{z(1-g(z))};z_j\right)dy\nonumber\\
&&= \mathrm{Res}\left(\frac{e^{-x z}}{z(1-g(z))} \left[ \frac{\kappa}{z+\kappa} - \frac{\alpha}{c} \frac{\kappa}{z+\kappa} m -1 \right];z_j\right)\nonumber\\
&&=   \mathrm{Res}\left(\frac{e^{-x z}}{z(1-g(z))} \frac{[-z - \frac{\alpha}{c} \kappa m]}{z+\kappa};z_j\right).\label{exp_I2_2}
\end{eqnarray}
Plugging (\ref{exp_I2_2}) in (\ref{exp_I2_1_bis}) and summing $I_2(x)$ and Expression (\ref{exp_I4_1}) of $I_4(x)$ then yields from (\ref{expansion_ruin1})
$$
\psi_\kappa(x)-C = \sum_{j=1}^N \mathrm{Res}\left(\frac{e^{-x z}}{1-g(z)} \frac{[-1+\frac{\alpha}{c}   m]}{z+\kappa};z_j\right)+ o(e^{-rx}).
$$
Let us now note that in the case where $g'(z_j)\neq 0$, the residue in the $j$th term in the above can be explicitly given and, thanks to (\ref{g(zj)_prime_ruin}), is equal to (\ref{g_prime_not_zero_continuous_risk}). Let us also note that expression of $C$ provided in (\ref{constant_C}) is but $-\frac{-1+\frac{\alpha}{c}   m}{ \frac{\alpha}{c} \EE \left( Ze^{(z_0+\kappa)Z}\right)-1}$ with $z_0=0$ and that $g'(z_0)\neq 0$, so that the above yields Expansion (\ref{exp_ruin_new}).
\end{proof}

\subsection{Proof of Theorem~\ref{thm:ruin_discret}}
We let $b(k):=\PP(Z_1=k)$, $k\in\nbN$, the probability mass distribution of $Z_1$ and denote its survival function by $ l(k):=\PP[Z_1\ge k+1]=\sum_{j=k+1}^\infty b(k)$.
Let us in particular note that $b(0)$ can be positive. In other words, there is a possibility that there is no claim (or equivalently, a claim of size $0$) at time $n$. We define
$$
{\cal D}_d:=\{ f:\nbN\longrightarrow \nbC|\ f(k)=o(e^{-sk})\ \forall s>0 \}
$$
and suppose here that the $Z_k$'s verify that $k\mapsto l(k) \in {\cal D}_d$ .
we let $\kappa>0$ its unique real positive solution.

 The following lemma is the discrete analog of Lemma \ref{lemme_D}:
\begin{lemm}\label{lemme_D_d}
Let $f\in {\cal D}_d$, then
\begin{enumerate}
	\item $k\in\nbN\mapsto e^{zk}f(k)$ belongs to ${\cal D}_d$ for all $z\in\nbC$,
	\item $k\in\nbN\mapsto \displaystyle \sum_{j=k}^\infty f(j)$ belongs to ${\cal D}_d$.
\end{enumerate}
\end{lemm}
Similarly to (\ref{expression_L(x)}), (\ref{Lundberg}) and (\ref{def_g_ruin}) we define the following discrete measures $F(.)$, associated probability mass function $f(.)$ and complex valued function $g(.)$ which we recall here:
\begin{eqnarray}
F(x)&:=& \sum_{k=0}^x f(k),\ x\in\nbN,\quad f(k):=e^{\kappa k}l(k)\nonumber\\
&=&e^{\kappa k} \PP(Z\ge k+1),\ k\in\nbN ,\label{def_f_ruin_discrete}\\
g(z) &:=& \int_0^\infty e^{z x}dF(x) = \sum_{k=0}^\infty e^{(z+\kappa) k}\PP(Z\ge k+1)\nonumber\\
&=&\frac{1-\EE(e^{(z+\kappa)Z})}{1-e^{z+\kappa}},\quad z\in\nbC .\label{def_g_ruin_discrete}
\end{eqnarray}
$g(z)$ is thus defined for all $z\in \nbC$ and is equal to $\EE (e^{zX})$ for some integer valued r.v. $X$ with probability mass function $f(.)$, and with expectation and second moment $\EE(X):=\mu$ and $\EE(X^2):=\mu_2$.

\begin{proof}[Proof of Theorem~\ref{thm:ruin_discret}]
We proceed along the lines of Proof of Proposition \ref{prop_exp_ruin}. Setting $\bar{L}(x):=\sum_{y=x}^\infty l(y)$, $x\in\nbN$, the discrete time analog of (\ref{renewal_eq}) is
\begin{equation}
\psi(x)=\sum_{k=0}^x \psi(x-k)l(k) + \bar{L}(x),\quad x\in\nbN ,
\label{renewal_eq_discrete}
\end{equation}
see e.g. Proposition 1.2 p.488 in \cite{AA}. Letting $\psi_\kappa(x):=e^{\kappa x}\psi(x)$ and $z(x):=e^{\kappa x} \bar{L}(x)$, and by the definition of $f(k)$ for all $k$ in (\ref{def_f_ruin_discrete}), we obtain from (\ref{renewal_eq_discrete}) the renewal equation
\begin{equation}\label{renewal_eq_dieze_discrete}
\psi_\kappa(x)= z(x)+ \sum_{k=0}^x \psi_\kappa(x+1-k)f(k).
\end{equation}
Since $X$ is lattice, the corresponding Smith's renewal theorem implies that
\begin{equation}
\psi_\kappa(x)\longrightarrow C:=\frac{1}{\mu}\sum_{k=0}^\infty z(k)=\frac{m-e^\kappa}{e^\kappa-\EE(Ze^{\kappa Z})},\quad x\to\infty ,
\label{constant_C_discrete}
\end{equation}
the last equality can be verified by direct computation. Let us note in passing that (\ref{constant_C_discrete}) provides the Cramer Lundberg asymptotics $\psi(x)\sim \frac{m-e^\kappa}{e^\kappa-\EE(Ze^{\kappa Z})}e^{-\kappa x}$ for $x\to \infty$, $x\in\nbN$, in the discrete case. Defining now
$$
U(j):=\sum_{n=0}^\infty F^{\ast n}(j),\quad u_j:=U(j)-U(j-1),\quad v(j):=U(j) - \frac{j}{\mu}- \frac{\mu_2+\mu^2}{2 \mu^2},\quad j\in\nbN 
$$
with $U(-1)=0$, solution to (\ref{renewal_eq_dieze_discrete}) is given by
\begin{eqnarray*}
\psi_\kappa(x)&=&\sum_{k=0}^x z(x-k)u_k = \sum_{k=0}^x z(x-k) \left[\frac{1}{\mu}+ v(k)-v(k-1)\right],\\
&=& \frac{1}{\mu} \sum_{k=0}^x z(x-k) + \sum_{k=0}^x z(x-k) v(k) - \sum_{k=-1}^{x-1} z(x-k-1)v(k),\\
&=& \frac{1}{\mu} \sum_{k=0}^x z(k) +\sum_{k=0}^{x-1} [z(x-k)-z(x-k-1)]v(k) + z(0)v(x) - z(x) v(-1).
\end{eqnarray*}
Substracting in the above constant $C$ given by (\ref{constant_C_discrete}) yields
\begin{equation}
\psi_\kappa(x)-C= - \frac{1}{\mu} \sum_{k=x+1}^\infty z(k)+\sum_{k=0}^{x-1} [z(x-k)-z(x-k-1)]v(k) + z(0)v(x) - z(x) v(-1).
\label{expansion_ruin1_discrete}
\end{equation}
The first and last term on the right hand side of (\ref{expansion_ruin1_discrete}) are $o(e^{-rk})$ by Lemma \ref{lemme_D_d}, hence we are interested in the following quantities
\begin{eqnarray*}
I_1(x)&:=&\sum_{k=0}^{x-1} [z(x-k)-z(x-k-1)]v(k)=\sum_{k=1}^{x} [z(k)-z(k-1)]v(x-k),\\
I_2(x)&:=& z(0)v(x)=\bar{L}(0) v(x)=m v(x).
 \end{eqnarray*}
Thanks to Theorem \ref{thm_expansion_v_lattice}, one has that
\begin{equation}
I_2(x)=\sum_{j=0}^N \mathrm{Res}\left(\frac{e^{-kz}}{(e^z-1)(1-g(z))};z_j\right)+ o(e^{-rx})
\label{I2_discrete}
\end{equation}
and we then turn to $I_1(x)$. Writing
\begin{eqnarray}
z(k)-z(k-1)&=&e^{\kappa k} \bar{L}(k)-e^{\kappa (k-1)} \bar{L}(k-1)\nonumber\\
&=&[1-e^{-\kappa}]e^{\kappa k}\bar{L}(k) +e^{\kappa (k-1)} [\bar{L}(k) - \bar{L}(k-1)]\nonumber\\
&=& [1-e^{-\kappa}]e^{\kappa k}\bar{L}(k) - e^{\kappa (k-1)} l(k-1)\nonumber\\
&=& [1-e^{-\kappa}] z(k)- e^{\kappa (k-1)} l(k-1),
\label{I1_disctete1}
\end{eqnarray}
then writing, in view of expansion (\ref{expansion_v_lattice}),
$$
v(x-k)= \sum_{j=1}^N \mathrm{Res}\left(\frac{e^{-(x-k)z}}{(e^z-1)(1-g(z))};z_j\right) + \varepsilon_{x-k} e^{-r(x-k)}
$$
for some $(\varepsilon_n)_{n\in\nbN}$ vanishing at $\infty$, one obtains, 
\begin{eqnarray} 
\quad \ \ I_1(x)&=& \sum_{j=1}^N \mathrm{Res}\left[ \frac{1}{(e^{z}-1)(1-g(z)}\sum_{k=1}^x [1-e^{-\kappa}] e^{\kappa k}\bar{L}(k) e^{-(x-k)z};z_j \right]\label{I1_disctete2}\\
&-& \sum_{j=1}^N \mathrm{Res}\left[ \frac{1}{(e^{z}-1)(1-g(z)}\sum_{k=1}^{x} e^{\kappa (k-1)} l(k-1) e^{-(x-k)z};z_j \right]\label{I1_disctete3}\\
&+& \sum_{k=1}^x [z(k)-z(k-1)] \varepsilon_{x-k} e^{-r(x-k)}.
\label{I1_disctete4}
\end{eqnarray}
Using Lemma \ref{lemme_D_d} as well as a dominated convergence theorem, it is not hard to see that (\ref{I1_disctete4}) is an $o(e^{-rx})$ as $x\to \infty$, $x\in\nbN$. We then study (\ref{I1_disctete2}) and (\ref{I1_disctete3}). One verifies that $x\in\nbN\mapsto \mathrm{Res}\left[\frac{1}{(e^{z}-1)(1-g(z)}\sum_{k=x+1}^{\infty} e^{\kappa (k-1)} l(k-1) e^{k z};z_j\right]\in {\cal D}_d$ by writing the residue as an integral over a fixed contour including $z_j$ then using Lemma \ref{lemme_D_d}, so that in (\ref{I1_disctete3}) we compute
\begin{eqnarray}
&&\mathrm{Res}\left[\frac{1}{(e^{z}-1)(1-g(z))}e^{-x z}\sum_{k=1}^{x} e^{\kappa (k-1)} l(k-1) e^{k z};z_j \right]\nonumber\\
&= &\mathrm{Res}\left[\frac{1}{(e^{z}-1)(1-g(z))}e^{-x z} e^{z}\sum_{k=0}^{\infty}e^{(z+\kappa) k}\PP(Z\ge k+1) ;z_j\right]\nonumber\\
&-&\mathrm{Res}\left[\frac{1}{(e^{z}-1)(1-g(z))} e^{-x z} \sum_{k=x+1}^{\infty} e^{\kappa (k-1)} l(k-1) e^{k z};z_j\right]\nonumber\\
&=& \mathrm{Res}\left[\frac{1}{(e^{z}-1)(1-g(z))} e^{-x z}e^{z}g(z);z_j\right]+ o(e^{-rx}),\quad j=1,...,N.
\label{I1_disctete5}
\end{eqnarray}
Next, we compute, for all $z$,
\begin{eqnarray}
\sum_{k=1}^\infty   e^{(\kappa +z) k}\bar{L}(k) &=& \sum_{k=1}^\infty \sum_{n=k}^\infty e^{(\kappa +z) k} l(n)=\sum_{n=1}^\infty \left[\sum_{k=1}^{n} e^{(\kappa +z) k}\right] l(n)\nonumber\\
&=& \sum_{n=1}^\infty \left[ e^{\kappa +z} \frac{1-  e^{(\kappa +z) n}}{1-e^{\kappa +z} }\right] l(n)\nonumber\\
&=& \frac{e^{\kappa +z} }{1-e^{\kappa +z} } \sum_{n=1}^\infty l(n)- \frac{1}{1-e^{\kappa +z}  }e^{\kappa+z} \sum_{n=1}^\infty   e^{(\kappa+ z) n} l(n)\nonumber\\
&=& \frac{e^{\kappa +z} }{1-e^{\kappa +z} } [m-l(0)] - \frac{1}{1-e^{\kappa +z}  }e^{\kappa+z} [g(z)-l(0)]\nonumber\\
&=& \frac{e^{\kappa +z}}{1-e^{\kappa +z}}[m-g(z)],
\label{I1_disctete6}
\end{eqnarray}
One verifies thanks to Lemma \ref{lemme_D_d} and writing the residue as a contour around $z_j$ that $x\in\nbN\mapsto \mathrm{Res}\left[\frac{1}{(e^{z}-1)(1-g(z)} \sum_{k=x+1}^\infty   e^{(\kappa +z) k}\bar{L}(k);z_j\right]\in {\cal D}_d$, so that inserting (\ref{I1_disctete6}) in (\ref{I1_disctete2}) yields
\begin{multline}
\mathrm{Res}\left[\frac{1}{(e^{z}-1)(1-g(z))} \sum_{k=1}^x [1-e^{-\kappa}] e^{\kappa k}\bar{L}(k) e^{-(x-k)z};z_j\right]\\
=\mathrm{Res}\left[\frac{1}{(e^{z}-1)(1-g(z))} e^{-x z}\frac{e^{\kappa +z}}{1-e^{\kappa +z}}[m-g(z)];z_j\right]+ o(e^{-rx}).
\label{I1_discrete30}
\end{multline}
Respectively inserting (\ref{I1_disctete5}) and (\ref{I1_discrete30}) in (\ref{I1_disctete3}) and (\ref{I1_disctete2}) in $I_1(x)$, and adding $I_2(x)$ in (\ref{I2_discrete}), yields
\begin{eqnarray}
I_1(x)+I_2(x)&=& \sum_{j=1}^N \mathrm{Res}\left[\frac{1-e^{-\kappa}}{(e^{z}-1)(1-g(z))}  \frac{e^{\kappa +z}}{1-e^{\kappa +z}}[m-g(z)]  e^{-x z};z_j\right]\nonumber\\
&-& \sum_{j=1}^N \mathrm{Res}\left[ \frac{g(z)}{(e^{z}-1)(1-g(z))} e^{z} e^{-x z};z_j\right]\nonumber\\
&+& \sum_{j=0}^N \mathrm{Res} \left[ \frac{m}{(e^{z}-1)(1-g(z))}e^{-x z};z_j\right]+ o(e^{-rx}) ,\nonumber\\
&=& -\sum_{j=1}^N \mathrm{Res} \left[\frac{1}{1-g(z)} \frac{m-g(z)e^{\kappa+z}}{1-e^{\kappa+z}} e^{-x z};z_j\right]+o(e^{-rx}). \label{I1_I2_discrete}
\end{eqnarray}
Plugging (\ref{I1_I2_discrete}) into (\ref{expansion_ruin1_discrete}) yields (\ref{exp_ruin_new}), provided that we prove that residue of $-\frac{1}{1-g(z)} \frac{m-g(z)e^{\kappa+z}}{1-e^{\kappa+z}} e^{-x z}$ at $z=z_0=0$ equals $Ce^{-\kappa x}$, $C$ given by (\ref{constant_C_discrete}). This computation is in fact included in the case where the $g'(z_j)$'s are non zero, in which case $j$th term of (\ref{I1_I2_discrete}) is equal to (\ref{expression_term_ruin_discrete}), which we proceed to consider now. Direct computation from (\ref{def_g_ruin_discrete}) coupled to the fact that $g(z_j)=1$ gives
\begin{multline*}
g'(z_j)= \frac{-\EE\left(Z e^{(\kappa+z_j)Z}\right)[1- e^{\kappa+z_j}]+[1-\EE(e^{(\kappa+z_j)Z})]e^{\kappa+z_j}}{(1- e^{\kappa+z_j})^2}\\
=-\frac{\EE\left(Z e^{(\kappa+z_j)Z}\right)}{1- e^{\kappa+z_j}}+ \frac{e^{\kappa+z_j}}{1- e^{\kappa+z_j}}
\end{multline*}
so that it is easy to verify that $\mathrm{Res} \left[\frac{1}{1-g(z)} \frac{m-g(z)e^{\kappa+z}}{1-e^{\kappa+z}} e^{-x z};z_j\right]$ is (\ref{expression_term_ruin_discrete}) times $e^{-(\kappa+z_j)x}$.

\end{proof}

\subsection{Proof of Proposition~\ref{Prop_2dim}}
\begin{proof}[Proof of Proposition~\ref{Prop_2dim}]
We recall that two functions $f$ and $h$ satisfy $f(x)\gg h(x)$ iff $h(x)/f(x)$ tends to $0$ as $x\to\infty$. In the following we consider cases where $q>0$ is such that functions respectively satisfy
\begin{eqnarray}
&& C_0^1 e^{-\kappa_1 x} \gg \Re\left[C_1^1 e^{-(\kappa_1+z^1_1) x}\right] \gg \left\{C_0^2 e^{-q\kappa_2 x} ,\ \Re\left[C_1^2 e^{-q(\kappa_2+z^2_1) x}\right]\right\},\label{case1_2dim}\\
&& C_0^1 e^{-\kappa_1 x} \gg C_0^2 e^{-q\kappa_2 x} \gg \left\{\Re\left[C_1^1 e^{-(\kappa_1+z^1_1) x}\right],\ \Re\left[C_1^2 e^{-q(\kappa_2+z^2_1) x}\right]\right\},\label{case2_2dim}
\end{eqnarray}
all other cases being treated similarly. We will use the following inequalities (an easy and direct consequence of Lemma 2 (i) of \cite{APP})
\begin{equation}
\psi_1(x_1)\le \psi_{\small \mbox{or}}(x_1,x_2) \le \psi_1(x_1)+\psi_2(x_2).
\label{ineq_APP}
\end{equation}
Let us first consider the case (\ref{case1_2dim}). Using expansions (\ref{two_dim_expansions}) with $x_1=x$, $x_2=qx$, as well as (\ref{ineq_APP}) we get
\begin{multline}
\varepsilon_1(x)e^{-(r+\kappa_1)x}\le \psi_{\small \mbox{or}}(x,qx)-  C_0^1 e^{-\kappa_1 x}
-\Re\left[C_1^1 e^{-(\kappa_1+z^1_1) x}\right] 
\\\le C_0^2 e^{-q\kappa_2 x} + \Re\left[C_1^2 e^{-q(\kappa_2+z^2_1) x}\right] + \varepsilon_1(x)e^{-(r+\kappa_1)x}+ \varepsilon_2(qx)e^{-(r+\kappa_1)qx}.\label{expansion_2dim_1}
\end{multline}
We now note that the terms on the far left and right hand side of (\ref{expansion_2dim_1}) divided by $e^{-\Re (\kappa_1+z_1^1)x}$ tend to $0$ because of Assumption (\ref{case1_2dim}), which proves~(\ref{two_terms_two dim}).\\
We now consider case (\ref{case2_2dim}). Using again (\ref{ineq_APP}) we get, similarly to (\ref{expansion_2dim_1}),
\begin{multline}
\varepsilon_1(x)e^{-(r+\kappa_1)x} + \Re\left[C_1^1 e^{-(\kappa_1+z^1_1) x}\right] - C_0^2 e^{-q\kappa_2 x} \le \psi_{\small \mbox{or}}(x,qx)-  C_0^1 e^{-\kappa_1 x} - C_0^2 e^{-q\kappa_2 x} 
\\\le \Re\left[C_1^1 e^{-(\kappa_1+z^1_1) x}\right]  + \Re\left[C_1^2 e^{-q(\kappa_2+z^2_1) x}\right] + \varepsilon_1(x)e^{-(r+\kappa_1)x}+ \varepsilon_2(qx)e^{-(r+\kappa_1)qx}.\label{expansion_2dim_2}
\end{multline}
Because of Assumption (\ref{case2_2dim}), we get from (\ref{expansion_2dim_2}) that
\begin{multline*}
-1=\lim_{x\to\infty}\frac{\varepsilon_1(x)e^{-(r+\kappa_1)x} + \Re\left[C_1^1 e^{-(\kappa_1+z^1_1) x}\right] - C_0^2 e^{-q\kappa_2 x}}{C_0^2 e^{-q\kappa_2 x}}\\
\le \liminf_{x\to\infty}\frac{\psi_{\small \mbox{or}}(x,qx)-  C_0^1 e^{-\kappa_1 x} - C_0^2 e^{-q\kappa_2 x}}{C_0^2 e^{-q\kappa_2 x}}\\
\le \limsup_{x\to\infty}\frac{\psi_{\small \mbox{or}}(x,qx)-  C_0^1 e^{-\kappa_1 x} - C_0^2 e^{-q\kappa_2 x}}{C_0^2 e^{-q\kappa_2 x}}\\
\le \lim_{x\to\infty}\frac{\Re\left[C_1^1 e^{-(\kappa_1+z^1_1) x}\right]  + \Re\left[C_1^2 e^{-q(\kappa_2+z^2_1) x}\right] + \varepsilon_1(x)e^{-(r+\kappa_1)x}+ \varepsilon_2(qx)e^{-(r+\kappa_1)qx}}{C_0^2 e^{-q\kappa_2 x}}
\\=0,
\end{multline*}
which proves (\ref{two_terms_two dim}).
\end{proof}
{\bf Acknowledgment.} The authors wish to thank an anonymous referee for pointing out the informal proof given at the end of Section \ref{sec:main_results}, comments leading to Remark \ref{comment_referee}, as well as reference \cite{KP} leading to the infinite expansion for $\psi(x)$ in Example \ref{ex_stop_loss}. This work was partially supported by the French National Agency of Research ANR Grant, project AMMSI, number ANR-2011-BS01-021.

\end{document}